\documentclass[12pt]{article}

\usepackage[utf8]{inputenc}
\usepackage{float}
\usepackage{comment}
\usepackage{adjustbox}
\usepackage{booktabs}

\usepackage{graphicx}
\usepackage[top=1.25in, bottom=1.25in, left=1.5in, right=1.25in]{geometry}
\newcommand{\R}{\mathbb{R}}
\newcommand{\C}{\mathbb{C}}
\newcommand{\eat}[1]{}

\usepackage{enumitem}
\setlist{nolistsep,noitemsep}

\usepackage{amsmath}
\usepackage[mathscr]{euscript}
\usepackage{amssymb, nccmath}
\usepackage{amsthm}
\usepackage{array,multirow}

\usepackage[noadjust, space]{cite}
\usepackage{multicol}
\usepackage{comment}
\usepackage{tikz}
\usetikzlibrary{overlay-beamer-styles}
\usetikzlibrary{arrows,shapes}  
\usetikzlibrary{positioning,chains,fit,calc}
\tikzset{fontscale/.style = {font=\relsize{#1}}}
\tikzset{main node/.style={circle,fill=blue!20,draw,minimum size=1cm,inner sep=0pt}}
\usepackage[american]{circuitikz}

\usepackage{relsize}
\usepackage{caption}
\usepackage{subcaption}

\newtheorem{theorem}{Theorem}[section]

\newtheorem{lemma}[theorem]{Lemma}
\newtheorem{proposition}[theorem]{Proposition}
\newtheorem{conjecture}[theorem]{Conjecture}
\newtheorem{definition}[theorem]{Definition}
\newtheorem{remark}[theorem]{Remark}

\newtheorem{observation}[theorem]{Observation}

\newcommand{\Rn}{{\mathbb{R}^n}}
\newcommand{\lmin}{\lambda_{\min}}

\newcommand{\Z}{{\mathbb{Z}}}

\newcommand{\Rnn}{\R^{n\times n}}

\newcommand{\Lmod}{\tilde{L}}

\newcommand{\Lmodd}{\overset{\approx}{L}}

\newcommand{\Cchi}{\protect\raisebox{2pt}{$\chi$}}
\newcommand{\Pp}{P_n^*}
\newcommand{\Pvar}{T_{2n}}

\renewcommand \thesection{\arabic{section}}
\renewcommand \thesubsection{\arabic{section}.\arabic{subsection}}

\renewcommand{\baselinestretch}{1.5}

\usepackage[normalem]{ulem}
\usepackage{accents}
\newcommand*{\dt}[1]{%
	\accentset{\mbox{\large\bfseries .}}{#1}}

\allowdisplaybreaks

\makeatletter
\def\endthebibliography{%
\def\@noitemerr{\@latex@warning{Empty `thebibliography' environment}}%
\endlist
}
\makeatother

\begin{document}

\title{\LARGE \bf Optimal $k$-centers of a graph: a control-theoretic approach}

\author{Karim Shahbaz$^{\dagger}$, Madhu N. Belur$^{\dagger}$, Chayan Bhawal\thanks{Chayan Bhawal (bhawal@iitg.ac.in) is
in the Department of Electronics and Electrical Engineering, Indian Institute of Technology Guwahati, India.}, Debasattam Pal\thanks{Karim Shahbaz (karimshah05@gmail.com),
Madhu N. Belur (belur@ee.iitb.ac.in) and Debasattam Pal (debasattam@ee.iitb.ac.in ) are in the Department of Electrical Engineering, Indian Institute of Technology Bombay, India.}}


\maketitle

\begin{abstract}
In a network consisting of $n$ nodes, our goal is to identify the `most central’ $k$ nodes with respect to the proposed definitions of centrality. This concept finds applications in various scenarios, ranging from multi-agent problems where external communication (or leader nodes) must be identified, to more classical use cases like ambulance or facility location problems. Depending on the specific application, there exist several metrics for quantifying $k$-centrality, and the subset of the best $k$ nodes naturally varies based on the chosen metric. Closely related to graphs, we also explore notions of $k$-centrality within stochastic matrices. In this paper, we propose two metrics and establish connections to a well-studied metric from the literature (specifically for stochastic matrices). We prove these three  notions match for path-graphs. We then list a few more control-theoretic notions and compare these various notions for general randomly generated graphs. Our first metric involves maximizing the shift in the smallest eigenvalue of the Laplacian matrix. This shift can be interpreted as an improvement in the time constant when the RC circuit experiences leakage at certain $k$ capacitors. The second metric focuses on minimizing the Perron root of a principal sub-matrix of a stochastic matrix-an idea proposed and interpreted in the literature as `manufacturing consent'. The third one explores minimizing the Perron root of a perturbed (now super-stochastic) matrix, which can be seen as minimizing the impact of added stubbornness. It is important to emphasize that we consider applications (for example, facility location) when the notions of central ports are such that the set of the `best' $k$ ports does not necessarily contain the set of the $k$-1 ports. We apply our $k$-port selection metric to various network structures. Notably, we prove the equivalence of three definitions for a path graph and extend the concept of central port linkage beyond Fiedler vectors to other eigenvectors associated with path graphs.
\end{abstract}
%

\section{Introduction}\label{sec:intro}

Optimal port selection involves determining the most appropriate or ideal ports within a system or network to serve particular purposes.  This process pinpoints one or more highly significant nodes in the network, determined by a criterion of centrality measures. Optimal port selection has wide-ranging applications in fields such as networking, electrical circuits, logistics, transportation, data centers, robotics, and automation. Its significance extends to efficient/fastest communication, network resilience, resource allocation, security, influence and control.

In this paper, we consider only simple undirected, unweighted graphs, i.e., with no self-loops
and no multiple edges between any pairs of the vertices. Node centrality is a well-explored topic in the literature. Various centrality measures, including degree centrality \cite{ShawDegree}, closeness centrality \cite{SabidussiCloseness}, betweenness centrality \cite{AnthonisseBetween}, eigenvector centrality \cite{BonacichEigenCentral}, and Page-Rank \cite{PageRank}, are used to rank nodes based on different aspects of the network. In most of these centrality measures, the set of `best' $k$ ports, denoted as $S_k$, includes the `best' $k$-1 ports, set $S_{k-1}$. However, in our concept of centrality, the inclusion of $S_{k-1}$ in $S_k$ is not assumed. For example, consider Figure~\ref{fig:ChoosingInGraph}. It shows a graph of $11$ nodes and its corresponding RC network with resistance $r_c$ representing edges and nodes with capacitance $c$. It shows best one center node $5$  and best two centers nodes $3$ and $8$ w.r.t. to Defn~\ref{def:optim:all3}[i](MP\textbf{L}SE). Thus, the best one is not included in the best two centers. 

We propose two new notions of choosing the best $k$ ports within a network and compare with an existing one. One of the proposed notions quantifies the change in the smallest eigenvalue of the Laplacian matrix when the matrix is subjected to perturbation i.e. Maximized Perturbed Laplacian matrix Smallest Eigenvalue (\textit{Defn.~\ref{def:optim:all3}[i](MP\textbf{L}SE)}). Subsequently, we compare this selection with other metrics: Remarks~\ref{rem:interpretation:Submat} and \ref{rem:interpretation:Supmat} gives an interpretation of these metrics. 
\begin{itemize}
	\item \cite{BorkarNair} Minimized Sub-stochastic matrix Largest Eigenvalue based selection (\textit{Defn.~\ref{def:optim:all3}[ii](M\textbf{Sub}-LE)})
	\item Minimized Super-stochastic matrix Largest Eigenvalue based selection (\textit{Defn.~\ref{def:optim:all3}[iii](M\textbf{Sup}-LE)})
\end{itemize}

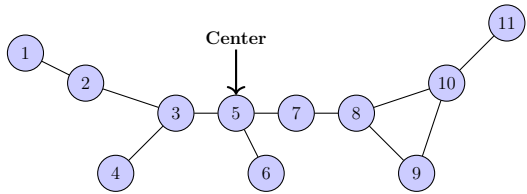
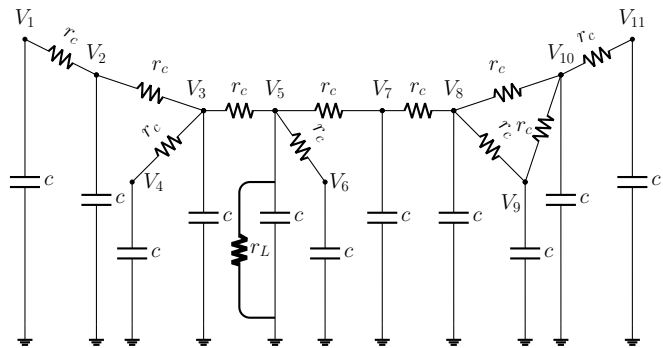
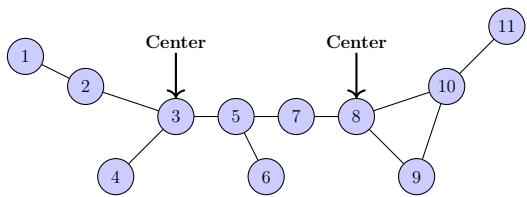
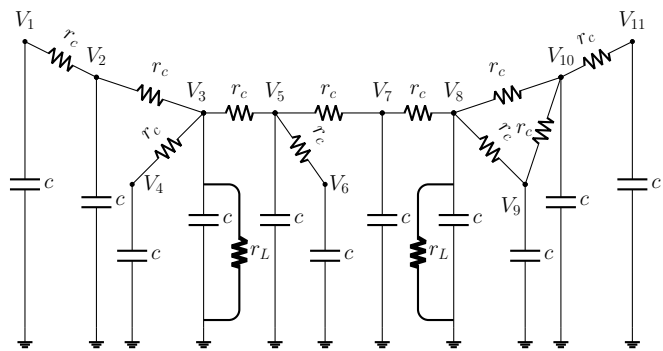
\begin{figure}[h!]
	\centering
	\captionsetup[subfigure]{margin=0.05\columnwidth}
	\begin{subfigure}{0.45\linewidth}\label{fig:G-center-1} 
		\centering
		\begin{tikzpicture} [scale=0.8, transform shape, every node/.style={scale=0.6}]
			\node[font = {\large},main node] (1) at (-1.5,1) {1};
			\node[font = {\large},main node] (2) at (-0.5,0.5) {2};
			\node[font = {\large},main node] (3) at (1,0) {3};
			\node[font = {\large},main node] (4) at (0,-1) {4};
			\node[font = {\large},main node] (5) at (2,0) {5};
			\node[font = {\large},main node] (6) at (2.5,-1) {6};
			\node[font = {\large},main node] (7) at (3,0) {7};
			\node[font = {\large},main node] (8) at (4,0) {8};
			\node[font = {\large},main node] (9) at (5,-1) {9};
			\node[font = {\large},main node] (10) at (5.5,0.5) {10};
			\node[font = {\large},main node] (11) at (6.5,1.5) {11};
			\path [-] (1) edge node {} (2);
			\path [-] (2) edge node {} (3);
			\path [-] (3) edge node {} (4);
			\path [-] (3) edge node {} (5);
			\path [-] (5) edge node {} (6);
			\path [-] (5) edge node {} (7);
			\path [-] (7) edge node {} (8);
			\path [-] (8) edge node {} (9);
			\path [-] (8) edge node {} (10);
			\path [-] (9) edge node {} (10);
			\path [-] (10) edge node {} (11);
			\node[font = {\large},] (14) at (2,1.25){\textbf{Center}};
			\path[thick, ->] (14) edge node {} (5);
		\end{tikzpicture} 
		\subcaption{\footnotesize Graph G with one center $5$ (center w.r.t. \textit{Defn.~\ref{def:optim:all3}[i](MP\textbf{L}SE))}} 
	\end{subfigure}		
	\begin{subfigure}{0.54\linewidth} \label{fig:RCNetwork_G-center-1} 
		\centering
		\ctikzset{bipoles/resistor/height=0.3}
		\ctikzset{bipoles/resistor/width=0.5}
		\begin{tikzpicture}[scale=0.95,transform shape, font = {\Large}, every node/.style={scale=0.5}]
			\draw (-1.5,1) to[R=$r_c$ , *-*] (-0.5, 0.5) to[R=$r_c$ , *-*] (1,0) to[R=$r_c$ , *-*] (2, 0) to[R=$r_c$ , *-*] (3.5, 0) to[R=$r_c$ , *-*] (4.5, 0) to[R=$r_c$ , *-*] (6, 0.5) to[R=$r_c$ , *-*] (7, 1); 
			\draw (0,-1) to[R=$r_c$ , *-*] (1,0) ; 
			\draw (2,0) to[R=$r_c$ , *-*] (2.7,-1) ; 
			\draw (4.5,0) to[R=$r_c$ , *-*] (5.5,-1) ;
			\draw (5.5,-1) to[R=$r_c$ , *-*] (6,0.5) ;
			
			\draw[rounded corners, thick] (2,-1) -- (1.5,-1) to[R=$r_L$] (1.5,-2.9) -- (2,-2.9);
			
			\draw (-1.5,1) to[C=$c$ ] (-1.5,-3);
			\draw (-0.5,0.5) to[C=$c$ ] (-0.5,-3);
			\draw (0,-1) to[C=$c$ ] (0,-3);
			\draw (1,0) to[C=$c$ ] (1,-3);
			\draw (2,0) to[C=$c$ ] (2,-3);
			\draw (2.7,-1) to[C=$c$ ] (2.7,-3);
			\draw (3.5,0) to[C=$c$ ] (3.5,-3);
			\draw (4.5,0) to[C=$c$ ] (4.5,-3);
			\draw (5.5,-1) to[C=$c$ ] (5.5,-3);
			\draw (6,0.5) to[C=$c$ ] (6,-3);
			\draw (7,1) to[C=$c$ ] (7,-3);
			
			
			\draw (-1.5,-3) node[ground]{};
			\draw (-0.5,-3) node[ground]{}; 
			\draw (0,-3) node[ground]{}; 
			\draw (1,-3) node[ground]{}; 
			\draw (2,-3) node[ground]{};
			\draw (2.7,-3) node[ground]{}; 
			\draw (3.5,-3) node[ground]{}; 
			\draw (4.5,-3) node[ground]{}; 
			\draw (5.5,-3) node[ground]{}; 
			\draw (6,-3) node[ground]{}; 
			\draw (7,-3) node[ground]{};

			\node[] at (-1.5,1.3) {$V_1$};
			\node[] at (-0.5,0.8) {$V_2$};
			\node[] at (0.9,0.3) {$V_3$};
			\node[] at (0.3,-1) {$V_4$};
			\node[] at (2,0.3) {$V_5$};
			\node[] at (2.9,-1) {$V_6$};
			\node[] at (3.5,0.3) {$V_7$};
			\node[] at (4.5,0.3) {$V_8$};
			\node[] at (5.3,-1.3) {$V_9$};
			\node[] at (6,0.8) {$V_{10}$};
			\node[] at (7,1.3) {$V_{11}$};
			
			\ctikzset{resistor = american}
		\end{tikzpicture} 
		\subcaption{\footnotesize An analogous RC network of G and leakage resistances across capacitors at node $5$ and the ground.} 
	\end{subfigure}
	
	\begin{subfigure}{0.45\linewidth}\label{fig:G-centers-2} 
		\centering
		\begin{tikzpicture} [scale=0.8, transform shape, every node/.style={scale=0.6}]
			\node[font = {\large},main node] (1) at (-1.5,1) {1};
			\node[font = {\large},main node] (2) at (-0.5,0.5) {2};
			\node[font = {\large},main node] (3) at (1,0) {3};
			\node[font = {\large},main node] (4) at (0,-1) {4};
			\node[font = {\large},main node] (5) at (2,0) {5};
			\node[font = {\large},main node] (6) at (2.5,-1) {6};
			\node[font = {\large},main node] (7) at (3,0) {7};
			\node[font = {\large},main node] (8) at (4,0) {8};
			\node[font = {\large},main node] (9) at (5,-1) {9};
			\node[font = {\large},main node] (10) at (5.5,0.5) {10};
			\node[font = {\large},main node] (11) at (6.5,1.5) {11};
			\path [-] (1) edge node {} (2);
			\path [-] (2) edge node {} (3);
			\path [-] (3) edge node {} (4);
			\path [-] (3) edge node {} (5);
			\path [-] (5) edge node {} (6);
			\path [-] (5) edge node {} (7);
			\path [-] (7) edge node {} (8);
			\path [-] (8) edge node {} (9);
			\path [-] (8) edge node {} (10);
			\path [-] (9) edge node {} (10);
			\path [-] (10) edge node {} (11);
			\node[font = {\large},] (12) at (1,1.25){\textbf{Center}};
			\node[font = {\large},] (13) at (4,1.25){\textbf{Center}};
			\path[thick, ->] (12) edge node {} (3);
			\path[thick, ->] (13) edge node {} (8);
		\end{tikzpicture} 
		\subcaption{\footnotesize Graph G with two centers $3$ $\&$ $8$ (centers w.r.t. \textit{Defn.~\ref{def:optim:all3}[i](MP\textbf{L}SE))}} 
	\end{subfigure}		
	\begin{subfigure}{0.54\linewidth} \label{fig:RCNetwork_G-centers-2} 
		\centering
		\ctikzset{bipoles/resistor/height=0.3}
		\ctikzset{bipoles/resistor/width=0.5}
		\begin{tikzpicture}[scale=0.95,transform shape, font = {\Large}, every node/.style={scale=0.5}]
			\draw (-1.5,1) to[R=$r_c$ , *-*] (-0.5, 0.5) to[R=$r_c$ , *-*] (1,0) to[R=$r_c$ , *-*] (2, 0) to[R=$r_c$ , *-*] (3.5, 0) to[R=$r_c$ , *-*] (4.5, 0) to[R=$r_c$ , *-*] (6, 0.5) to[R=$r_c$ , *-*] (7, 1); 
			\draw (0,-1) to[R=$r_c$ , *-*] (1,0) ; 
			\draw (2,0) to[R=$r_c$ , *-*] (2.7,-1) ; 
			\draw (4.5,0) to[R=$r_c$ , *-*] (5.5,-1) ;
			\draw (5.5,-1) to[R=$r_c$ , *-*] (6,0.5) ;
			
			\draw[rounded corners, thick] (1,-1) -- (1.5,-1) to[R=$r_L$] (1.5,-2.9) -- (1,-2.9);
			\draw [rounded corners, thick](4.5,-1) -- (4.0,-1) to[R=$r_L$] (4.0,-2.9)--(4.5, -2.9);
			\draw (-1.5,1) to[C=$c$ ] (-1.5,-3);
			\draw (-0.5,0.5) to[C=$c$ ] (-0.5,-3);
			\draw (0,-1) to[C=$c$ ] (0,-3);
			\draw (1,0) to[C=$c$ ] (1,-3);
			\draw (2,0) to[C=$c$ ] (2,-3);
			\draw (2.7,-1) to[C=$c$ ] (2.7,-3);
			\draw (3.5,0) to[C=$c$ ] (3.5,-3);
			\draw (4.5,0) to[C=$c$ ] (4.5,-3);
			\draw (5.5,-1) to[C=$c$ ] (5.5,-3);
			\draw (6,0.5) to[C=$c$ ] (6,-3);
			\draw (7,1) to[C=$c$ ] (7,-3);
			
			
			\draw (-1.5,-3) node[ground]{};
			\draw (-0.5,-3) node[ground]{}; 
			\draw (0,-3) node[ground]{}; 
			\draw (1,-3) node[ground]{}; 
			\draw (2,-3) node[ground]{};
			\draw (2.7,-3) node[ground]{}; 
			\draw (3.5,-3) node[ground]{}; 
			\draw (4.5,-3) node[ground]{}; 
			\draw (5.5,-3) node[ground]{}; 
			\draw (6,-3) node[ground]{}; 
			\draw (7,-3) node[ground]{};

			\node[] at (-1.5,1.3) {$V_1$};
			\node[] at (-0.5,0.8) {$V_2$};
			\node[] at (0.9,0.3) {$V_3$};
			\node[] at (0.3,-1) {$V_4$};
			\node[] at (2,0.3) {$V_5$};
			\node[] at (2.9,-1) {$V_6$};
			\node[] at (3.5,0.3) {$V_7$};
			\node[] at (4.5,0.3) {$V_8$};
			\node[] at (5.3,-1.3) {$V_9$};
			\node[] at (6,0.8) {$V_{10}$};
			\node[] at (7,1.3) {$V_{11}$};
			
			\ctikzset{resistor = american}
		\end{tikzpicture} 
		\subcaption{\footnotesize An analogous RC network of G and leakage resistances across capacitors at nodes $3,8$ and the ground.} 
	\end{subfigure}
	
	\caption{Choosing best $k$-ports in Graph G}\label{fig:ChoosingInGraph}
\end{figure}

\noindent This paper deals with both:
\begin{itemize}
	\item Laplacian matrices for unweighted, undirected graphs, and
	\item Symmetric stochastic matrices: a doubly stochastic version of the Laplacian matrix (see~\cite{Merris97}).
\end{itemize}  
 
\noindent In this context, we define $Z_n:= I -\tau L_n$, for a sufficiently
small $\tau > 0$, as in equation~\eqref{eq:tau:inequality}. It can be interpreted as 
a discretized version (a doubly stochastic version of the Laplacian matrix) of the continuous-time 
system $\dt{x}= -L x$, with $L$ as the Laplacian of the unweighted, undirected graph $G$.
For this paper we assume \vspace{-2mm}
\begin{equation} \label{eq:tau:inequality}
	Z_n:= I -\tau L_n \quad \mbox{where}\quad
	0<\tau < \frac{1}{\Delta (G)}, \quad \Delta(G):=\mbox{max. degree of the graph~} G.
\end{equation}

This assumption on $\tau$ ensures that $I-\tau L$ is a stochastic matrix: the exact value of $\tau$ is not relevant any more for the rest of this paper's results.
We do use the property that if $L$ is symmetric, tridiagonal and unreduced\footnote{A tridiagonal matrix is called \underline{unreduced} if all the super-diagonal and sub-diagonal entries are non-zeros.} then so is $Z_n$.

\subsection{Notation} \label{subsec:notation}
The notation we follow is standard: the sets of real and complex numbers
are denoted respectively by $\R$ and $\C$. $|S|$ is cardinality of set $S$. 
Given an undirected, unweighted simple\footnote{A graph $G$ is called simple if there are no self-loops and $G$ has at most one edge between any pair of nodes.} graph $G$, the number of vertices/nodes are usually $n$, out of which $k$ best nodes are to be selected: usually $k < n/2$. We index the nodes from $1$ to $n$.
The eigenvalues of the Laplacian matrix of the network graph $L(G)$ (or any symmetric matrix in $\Rnn$) are denoted
  as $\lambda_{1}\leqslant\lambda_{2}\leqslant\cdots\leqslant\lambda_{n}$, i.e.
  $\lambda_{1}=\lmin$ and  $\lambda_{n}=\lambda_{max}$.  
$Z$ is a doubly stochastic version of Laplacian matrix $L$. $\bar{n}$ is set of $n$ elements i.e. $\bar{n}:=\{1,2,\dots,n\}$. The minimum eigenvalue of the Laplacian matrix $\lambda_{\min}$ is central in this paper: its dependence is investigated w.r.t. the number of nodes, the perturbation $\epsilon$, port-index $p$ (or port-indices $p_1$ and $p_2$): depending on the context, we write explicitly only some of its arguments and do not write explicitly $\lambda_{\min}(n, \epsilon, p_1, p_2)$ each time. 
 
\subsection{Organization of this paper}
The paper is organized as follows: The subsequent section (Section~\ref{sec:definition}) defines the metrics that are pursued throughout the paper. This is succeeded by the main results (Section~\ref{sec:main:results}), which contains the principal findings of the study pertaining to the path graph. Subsequently, the auxiliary results (Section~\ref{sec:auxiliary:results}) is presented, which furnishes additional results that support or augment the main outcomes. This section also encompasses findings or results that hold true for general graphs. Following this, some heuristics (Section~\ref{sec:heuristics}) are presented for identifying central nodes in general graphs and a comparison with our metrics. The final section (Section~\ref{sec:conclusion}), the conclusion, encapsulates the findings of the paper and deliberates on their implications.

\section{Definitions of $k$-centrality and possible inter-relations}\label{sec:definition}
In the definition below, we propose/compare three metrics for best $k$-ports selection, one of which has been quite well-studied
in the literature; see \cite{BorkarNair, BorkarKarnikNairNalli} and also \cite{WangLiZhang21}: Defn~\ref{def:optim:all3}[B] below.

\begin{definition}\label{def:optim:all3}
Consider an undirected, unweighted simple graph $G_1$ of $n$ nodes and the corresponding Laplacian matrix $L_n\in \Rnn$. Let $Z_n\in \Rnn$ be the corresponding doubly stochastic matrix defined in
equation~\eqref{eq:tau:inequality}
For a subset $S \subset \{1,2, \dots, n\},$ with $|S|=k<n$, and for a sufficiently small
$\epsilon>0$,
construct\footnote{For sufficiently
    small positive $\epsilon$, the subset $S^*$ depends only on the Laplacian matrix $L$ and the graph. This
    too is elaborated later below.} 
matrices
$ ^S\tilde{L}_n$ and $^S\tilde{Z}_n$ as follows:
\begin{equation}  \label{eq:S:perturb:defns}
  ^S\tilde{L}_n:=L_n-\epsilon\sum_{j\in S}e_je_j^T
  \quad \mbox{ and } \quad
  ^S\tilde{Z}_n:=Z_n+\epsilon\sum_{j\in S}e_je_j^T.
\end{equation}
Also define the principal submatrix $^S\hat{Z} \in \R^{(n-k)\times (n-k)}$ obtained by 
   \underline{removing}\footnote{This is similar to the notion of grounded-Laplacian
   pursued in \cite{WangLiZhang21}.} the $k$ rows and $k$ columns
from $Z_n$ corresponding to the subset $S$.

\begin{itemize}
\item[{[i]}] Maximized Perturbed \textbf{Laplacian} Smallest Eigenvalue (MP\textbf{L}SE): The best $k$ nodes 
are defined as the subset $S^*_{MP\textbf{L}SE}$ 
  which maximizes $\lambda_{\min} (^S\Lmod_{n})$, i.e.
\[
  S^*_{MP\textbf{L}SE} = \arg~\underset{\underset{\mbox{with $|S|=k$}}{\mbox{all subsets $S$ }}}{\max}~\lambda_{\min} (^S\Lmod_{n}).
\]
\item[{[ii]}] \cite{BorkarNair, BorkarKarnikNairNalli} Minimized \textbf{Sub}-stochastic Largest Eigenvalue (M\textbf{Sub}-LE):
The best $k$ nodes 
is defined as the subset $S^*_{M\textbf{Sub}\mbox{-}LE}$ that minimizes
  $\lambda_{\max} ({}^S\hat{Z})$ of the principal
  submatrix ${}^S\hat{Z} \in \R^{(n-k)\times(n-k)}$ corresponding to $S$,
 with $k$ rows/columns indexed by $S$ \underline{removed} from $Z$:
\[
S^*_{M\textbf{Sub}\mbox{-}LE}: = \arg~\underset{\underset{\mbox{with $|S|=k$}}{\mbox{all subsets $S$ }}}{\min}~\lambda_{\max} (^S\hat{Z}).
\]
\item[{[iii]}] Minimized \textbf{Super}-stochastic Largest Eigenvalue (M\textbf{Sup}-LE): The best $k$ nodes 
is defined as the subset $S^*_{M\textbf{Sup}\mbox{-}LE}$: that minimizes the largest eigenvalue of the super-stochastic matrix 
 $^S\tilde{Z}_n \in \Rnn$ obtained by perturbation defined in
equation~(\ref{eq:S:perturb:defns}):
\[
S^*_{M\textbf{Sup}\mbox{-}LE} := \arg~\underset{\underset{\mbox{with $|S|=k$}}{\mbox{all subsets $S$ }}}{\min}~\lambda_{\max} (^S\tilde{Z}_n).
\]
\end{itemize}
\end{definition}

\noindent The following remark explains about possible non-uniqueness in the optimizing sets above.\vspace{-2mm}
\begin{remark}
Of course, for each of the optimization problems within the definition above, the argument of the maximization/minimization, need not be unique: for the set of `best' $k$ nodes to be well-defined, we assume uniqueness. In the specific results in this paper, we
assume appropriate conditions that indeed ensure the uniqueness of the argument. 
\end{remark}

\noindent The remark below provides the interpretation of Defn.~\ref{def:optim:all3}[i](MP\textbf{L}SE) as an RC circuit quickly discharging through leakage resistances provided at the $k$-nodes.\vspace{-2mm} 
\begin{remark}\label{rem:interpretation:MPLSE}
The notion of best $k$-nodes in the sense of Defn.~\ref{def:optim:all3}[i](MP\textbf{L}SE) is 
best understood, in our opinion, using an RC circuit\footnote{The RC circuit state space representation is  $\dt{v}=-Lv$, with $v$: the capacitor voltages, $v(t)\in \mathbb{R}^n$ (assuming resistance and capacitance are appropriately normalized).} and interpreting the diagonal entries' perturbation
as providing a large resistance across capacitors indexed by $S$. Defn.~\ref{def:optim:all3}[i](MP\textbf{L}SE) formulates that those nodes are defined as central where the capacitor leakage causes the perturbed RC circuit to have the fastest discharge. One can easily surmise that leakage provided at a peripheral node would cause a slow and prolonged discharge of the capacitor total charge, while leakage provided at relatively central nodes would cause a faster discharge of the RC circuit.  
\end{remark}

\noindent The remark below provides another interpretation of Defn.~\ref{def:optim:all3}[i](MP\textbf{L}SE) and its comparison with the ambulance-location and the maximal covering location problem.\vspace{-2mm}
\begin{remark} \label{rem:independence:2ports:unintuitive}
Another interpretation of Defn.~\ref{def:optim:all3}[i](MP\textbf{L}SE) that is related, but not equivalent to, is the so-called
`ambulance-location' problem that is well-studied in the facility location
and logistics literature: see \cite{Munch23:etal} for a recent paper and its references,
and also \cite{Church:74} for a classic paper initiating such work as
a maximal covering location problem. What is distinct and fairly unintuitive 
is that when choosing \underline{two} central locations for ambulance locations, then the choice of one location surprisingly does not play a role in the (local/global) optimality of the other location.
Remark~\ref{rem:RC:initial:voltage:reconciliation} later provides a reason for this fairly unintuitive feature of Defn.~\ref{def:optim:all3}[i](MP\textbf{L}SE).
\end{remark}
\noindent Definition \ref{def:optim:all3}[ii](M\textbf{Sub}-LE) is familiar and well-studied in the literature \cite{BorkarNair}, the principal sub-stochastic matrix largest eigenvalue minimization based selection. The following remark provides an interpretation of Defn.~\ref{def:optim:all3}[ii](M\textbf{Sub}-LE) in the sense of `manufacturing consent' i.e. how fast influence can be diffused.\vspace{-2mm}
\begin{remark}\label{rem:interpretation:Submat}
The notion of best $k$-nodes for a stochastic matrix in the sense of `manufacturing consent' (i.e.  Defn.~\ref{def:optim:all3}[ii](M\textbf{Sub}-LE))
 was proposed in \cite{BorkarNair, BorkarKarnikNairNalli} and is also closely related to that of the `grounded Laplacian'
(see \cite{Miekkala93}). In this definition, given a stochastic matrix $Z_n$ and given a subset $S$ of cardinality $k$, the corresponding $k$ rows and $k$ columns
are removed thus making the resulting matrix a sub-stochastic (or positive definite, in the case
of the grounded Laplacian). The minimization of the largest eigenvalue of the principal submatrix $^S\!\hat{Z} \in \R^{(n-k)\times (n-k)}$ is interpreted as, loosely speaking, how quickly the influence through the $k$ nodes over-rides onto other rest $n-k$ nodes compared to the influence due to the $n-k$ nodes on themselves. Further, the minimization of the largest eigenvalue implies faster convergence to steady state or equilibrium in stochastic process analogy. 
\end{remark}

\noindent The following remark discusses an interpretation of Defn.~\ref{def:optim:all3}[iii](M\textbf{Sup}-LE) as stubbornness/influence diffusion measure.\vspace{-2mm}
\begin{remark}\label{rem:interpretation:Supmat}
The Defn.~\ref{def:optim:all3}[iii](M\textbf{Sup}-LE) can be understood as follows. The addition of a small and positive value
of $\epsilon$ at diagonal elements of $Z$ corresponding to $S$ may be interpreted as introducing a small amount of `stubborn-ness' 
at those nodes. The addition of an amount $\epsilon>0$ to a diagonal entry can cause an increase of the largest eigenvalue
beyond 1 by varying amounts (but an amount between $0$ and $\epsilon$): the nodes that are quite central are able to `diffuse' this stubborn-ness (or 
influence) to and through their neighbouring nodes much better than nodes which are isolated or near to isolated. Minimizing
the largest eigenvalue of the super-stochastic matrix (as the proposed notion in Defn.~\ref{def:optim:all3}[iii](M\textbf{Sup}-LE)) amounts to finding nodes where the introduced stubbornness diffuses the most: thus more influential nodes can diffuse this stubborn-ness better.
\end{remark}

\noindent The following proposition states an expression for eigenvalues and the corresponding eigenvector of the path graph $P_n$ of order $n$. \vspace{-2mm}
\begin{proposition}\label{prop:pathEigvalEigvec}
	\cite[Lemma 6.6.1]{SpielmanBook} Let $L_n$ be the Laplacian matrix of the path graph $P_n$ of order $n$. \\
	Then the (ordered) eigenvalues $\lambda_{j}\in \mathbb{R}$ and corresponding eigenvectors, $v_j\in \mathbb{R}^{n}$ of the Laplacian matrix $L_n$ ($L_n v_j=\lambda_{j}v_j$ for $j=\{ 1,2, \dots,n\}$) satisfy the following:
	\begin{itemize}
		\item [i)] The eigenvalues $\lambda_{j}$ are distinct and  $0\leqslant\lambda_{j}<4$
		\item [ii)] 	The eigenvalue, $\lambda_j=2\biggl\{1-\cos\bigl(\pi (j-1)/n\bigr)\biggr\} $
		\item [iii)] The $p^{th}$ component of the corresponding eigenvector $v_j\in \mathbb{R}^{n}$,
		$v_j(p)=\cos\bigl(\pi (j-1)(p-0.5)/n\bigr) $ 
	\end{itemize}

\end{proposition}
\noindent The eigenvalues of the path graph are distinct as evident from the above proposition: this provides the uniqueness of the eigenvector (of course, up to scaling by a non-zero scalar).

\section{Main results: $k$ centrality and inter-relations}\label{sec:main:results}
In this section, we formulate the main results of this paper. The proofs of the results in this section need further results formulated and proved
in the following section. Our main results pertain to path graphs for which we prove that the optimality of the \underline{same set of $k$-central} nodes w.r.t. the different definitions listed in Defn.~\ref{def:optim:all3}[MP\textbf{L}SE, M\textbf{Sub}-LE, M\textbf{Sup}-LE]. Closed-form expressions are also formulated and
proved for the case of path graphs on $n$-nodes. For graphs that are more general than path graphs, the metrics and heuristics
do not agree 100\% and later in Section~\ref{sec:heuristics} we include computational experiments on randomly generated graphs
of various orders.

\subsection{``Central nodes" for path graph}

Our first main result obtains a single central node for the path graph w.r.t. all three Defn.~\ref{def:optim:all3}(MP\textbf{L}SE, M\textbf{Sub}-LE, M\textbf{Sup}-LE).\vspace{-2mm}
\begin{theorem}\label{thm:1port_path_Eigenshift}
	Let $n$ be odd, and consider the path graph $P_n$ on $n$ nodes and let $L_n$ be the Laplacian matrix. Define $Z_n := I_n -\tau L_n$ for 
	sufficiently small and positive $\tau$ (as in equation~\eqref{eq:tau:inequality}).
	Let $v_F$ be the Fiedler eigenvector, i.e. the eigenvector corresponding to $\lambda_2$.
	Consider the central node w.r.t. Defn.~\ref{def:optim:all3}[MP\textbf{L}SE, M\textbf{Sub}-LE, M\textbf{Sup}-LE].
	Define\footnote{\label{fn:physicalcenter}Thus, with $n$ odd, $p^*$ is the traditional physical center of the path graph $P_n$.} $p^*:=\cfrac{n+1}{2}$, then the following hold: 
	\begin{enumerate}
		\item The smallest eigenvalue $\lambda_{\min}$ of perturbed Laplacian matrix (up to $2^{nd}$ order in $\epsilon$) is maximized at $p=p^*$:\vspace{-4mm}
		\begin{equation}\label{eq:eigenshift:2perturbed}
			\begin{split}
				\underset{p}{\max}~\lambda_{min}(\Lmod_{n,\epsilon}, p)&=
				\cfrac{\epsilon}{n}-\cfrac{\epsilon^2}{4n}\sum_{j=2}^{n}\cfrac{\cos^2\left(\pi (j-1)/2\right)}{\sin^2\left(0.5\pi (j-1)/n \right)\left\{\sum_{p=1}^{n}\cos^2\left(\pi  (j-1)(p-0.5)/n\right)\right\}}\\
				&=\frac{\epsilon}{n} - \frac{\epsilon^2}{12n^2}(n^2-1)
			\end{split} 
		\end{equation}
		\item The node indexed at $p^*$ is center w.r.t. all 3 Definition \ref{def:optim:all3}[MP\textbf{L}SE, M\textbf{Sub}-LE, M\textbf{Sup}-LE](up to $2^{nd}$ order approx. in $\epsilon$).\\   
		\item \cite{BapatSukanta} Further, if $v_{2}\in \R^n$ is the eigenvector corresponding to the eigenvalue $\lambda_{2}$ (the Fiedler value) and $v_2(j)$ is the $j^{th}$ component of $v_2$, 
		then the $v_{2}(j)=0$ $\iff~~j=p^*$. 
	\end{enumerate}
\end{theorem}

\noindent See the appendix for the proof. An interesting fact that is the dependence on $p\in \R$ instead of $p \in \Z$. While $p^*=\frac{n+1}{2}$ is expected to be the maximizer for $\lambda_{\min}$ when $p \in \Z$, it is interesting that for $p \in \R$, the point $p^*$ is a \underline{local minima}: see Lemma~\ref{lem:1port_path_EigenShift}. Figure \ref{fig:1portfunction}  shows a plot of $\lambda_{\min}$ versus $p$ revealing the local minima, though $p$ is traditionally interpreted as port index (and an integer only). This is what makes the theorem significant; its proof uses Lemma~\ref{lem:1port_path_EigenShift} that formulates the dependence of $\lambda_{\min}$ on the port index more explicitly. \\
\begin{figure}[h!]
	\centering
	\includegraphics[width=0.6\textwidth, height=0.4\textwidth]{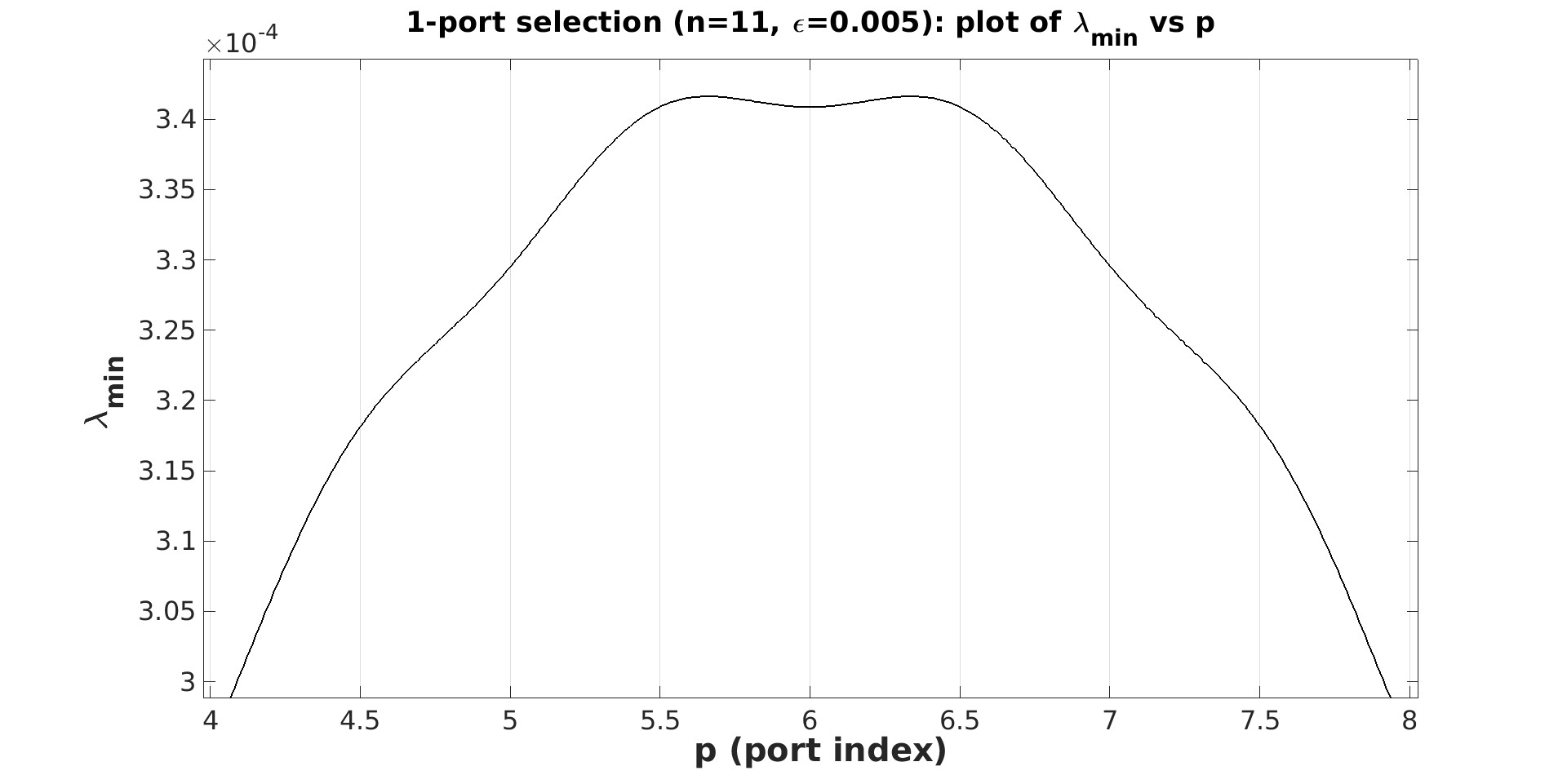}
	\caption{$\lambda_{\min}$ function vs $p$ plot for $P_{11}$ showing \underline{local minima} at the physical center.}
	\label{fig:1portfunction}
\end{figure}

The following result obtains \underline{two} central nodes for the path graph w.r.t. Defn.~\ref{def:optim:all3}[MP\textbf{L}SE, M\textbf{Sub}-LE, M\textbf{Sup}-LE]: here too the $3$ definitions coincide.\vspace{-2mm}
\begin{theorem}\label{thm:2ports_path_Eigenshift}
	Let $n$ be even and $\cfrac{n}{2}$ be odd, and consider the path graph $P_n$. Define $p_1^*:=\cfrac{n+2}{4}$ and $p_2^*:=\cfrac{3n+2}{4}$: the `physical centers' (see Footnote~\ref{fn:physicalcenter}) of the two halves of the path graph $P_n$. Then, for $j=\{1,2, \dots, n\}$ the following hold about the best $2$ central nodes w.r.t. Defn.~\ref{def:optim:all3}[MP\textbf{L}SE, M\textbf{Sub}-LE, M\textbf{Sup}-LE]. 
	\begin{enumerate}
		\item The smallest eigenvalue $\lambda_{\min}$ of the perturbed Laplacian matrix is maximized over all pairs ($p_1, p_2$) with $p_1,~p_2 \in \{1,2,\dots, n\}$  (for $2^{nd}$ order approximation in $\epsilon$): at $p_1=p_1^*$ and $p_2=p_2^*$
		\begin{equation}\label{eq:eigenshift:2perturbed}
			\begin{split}
				\underset{p_1,p_2}{\max}~ \lambda_{min}(\Lmodd_{n,\epsilon}, p_1, p_2)
				&=\lambda_{min}(\Lmodd_{n,\epsilon}, p_1, p_2)\Bigg |_{p_1=p_1^*, p_2=p_2^*}=\frac{2\epsilon}{n} -\frac{\epsilon^2}{12n^2}(n^2-4)\\
				&=\cfrac{2\epsilon}{n}-\cfrac{\epsilon^2}{4n}\sum_{j=2}^{n}\cfrac{\Bigl\{\cos\left(\pi (j-1)/4 \right)+\cos\left(3\pi (j-1)/4 \right)\Bigr\}^2}{\sin^2\left(0.5\pi (j-1)/n \right)\left\{\sum_{p=1}^{n}\cos^2\left(\pi  (j-1)(p-0.5)/n\right)\right\}}				
			\end{split}
		\end{equation}
		\item The nodes indexed at $p_1^*$ ~and~ $p_2^*$ are center w.r.t. all 3 Definitions \ref{def:optim:all3}[MP\textbf{L}SE, M\textbf{Sub}-LE, M\textbf{Sup}-LE].  \\ 
		\item Further, if $v_{3}$ is the eigenvector corresponding to the eigenvalue $\lambda_{3}$, and let $v_3(j)$ denotes the $j^{th}$ component of $v_3$. Then, for $j=\{1,2,\dots,n\}$, $v_{3}(j)=0$ $\iff~~j\in \{p_1^*,~p_2^*\}$. 
		
	\end{enumerate}
\end{theorem}

\noindent See the appendix for the proof. The following result obtains $k$ central nodes for the path graph w.r.t. Defn~\ref{def:optim:all3}[ii](M\textbf{Sub}-LE) and relates it to the eigenvector corresponding to $(k+1)^{th}$ smallest eigenvalue of the Laplacian matrix. This is generalization of part~(3) (the Fiedler vector based center) of Theorem~\ref{thm:1port_path_Eigenshift}\cite{BapatSukanta}. \vspace{-2mm}
\begin{theorem}\label{thm:kports_path}
	Let $n$ be divisible by $k$ and let $n/k$ be odd. Consider the path graph $P_n$. Define the $k$ nodes $p_i^*:=\cfrac{(2i-1)n+k}{2k}$ where $i=1, 2, \dots, k$: these are the physical centers of the $k$ equal parts of $P_n$. 
	
	\noindent Then, the following hold about the $k$ central nodes w.r.t. Defn.~\ref{def:optim:all3}[ii](M\textbf{Sub}-LE). \vspace{-2mm}
	\begin{enumerate}
		\item $p_1^*,~p_2^*, \dots, p_k^*$ form the unique central nodes w.r.t. Defn.~\ref{def:optim:all3}[i](MP\textbf{L}SE) and the corresponding optimal value of $\lambda_{\max}(^{S^*}\hat{Z})=\cos(k\pi/n)$.
		\item Let $v_{k+1}\in \R^n$ be the eigenvector corresponding to the eigenvalue $\lambda_{k+1}$ and let $v_{k+1}(j)$ denotes the $j^{th}$ component of $v_{k+1}$. Then, for $j=\{1,2, \dots, n\}$, $v_{k+1}(j)=0$ $\iff~~j\in \{p_1^*,~p_2^*, \dots, p_k^*\}$.
	\end{enumerate}
\end{theorem}
\noindent See the appendix for the proof. The following theorem states the equivalence of Defn.~\ref{def:optim:all3}[i](MP\textbf{L}SE) and [iii](M\textbf{Sup}-LE) for a general simple graph $G$.\vspace{-2mm}
\begin{theorem}\label{thm:superstochastic_eigenshift}
	Consider a general simple graph $G$ with $n$ nodes having Laplacian matrix $L_n$ and the stochastic matrix $Z_n=I_n-\tau L_n$. Let $|S|=k$. Let the Laplacian matrix $L_n$ be perturbed to $ ^S\tilde{L}_n:=L_n-\epsilon\sum_{j\in S}e_je_j^T$ and the stochastic matrix $Z_n$ be perturbed to $^S\tilde{Z}_n:=Z_n+\epsilon\sum_{j\in S}e_je_j^T$.
		 
	\noindent Then, the set $S_1^*$ that maximizes the smallest eigenvalue of perturbed Laplacian matrix by \\
	Defn.~\ref{def:optim:all3}[iii](M\textbf{Sup}-LE) and the set $S_2^*$ that minimizes the largest eigenvalue of $^S\!\tilde{Z}_n$ by\\ Defn.~\ref{def:optim:all3}[iii](M\textbf{Sup}-LE), are same\footnote{In case of non-uniqueness of the optimization in LHS/RHS in equation \eqref{eq:MPLSEequivSup}, the equation is to be understood as the sets which maximise $\lambda_{\min} (^S\Lmod_{n})$ is same as the sets which minimise $\lambda_{\max} (^S\!\tilde{Z}_n)$.} $k$-centers of the network i.e. 
	\begin{equation}\label{eq:MPLSEequivSup}
		\arg \underset{\underset{\mbox{with $|S|=k$}}{\mbox{all subsets S }}}{\max} \lambda_{\min} (^S\Lmod_{n})=\arg  \underset{\underset{\mbox{with $|S| =k$}}{\mbox{all subsets S }}}{\min}\lambda_{\max} (^S\!\tilde{Z}_n).
	\end{equation} 
\end{theorem}

\noindent See the appendix for the proof. 

\subsection{First order approximation of the perturbed Laplacian matrix smallest eigenvalue: \underline{general} graphs}
Theorems~\ref{thm:1port_path_Eigenshift}, \ref{thm:2ports_path_Eigenshift} and Lemma~\ref{lem:kports_path_EigenShift} reveal that the first order approximation of the smallest eigenvalue of the perturbed Laplacian matrix is $\cfrac{k\epsilon}{n}$ for a path graph: independent of the port indices at which perturbation is carried out. Extending this result to tree or general graph, we formulate the following theorem. \vspace{-2mm}
\begin{theorem}\label{thm:firstorder:general}
	Let $S \subset \{1,2,\dots, n\}$ for $|S|=k<n$. Let $L_n$ be the Laplacian matrix for any undirected, unweighted, simple and connected graph $G_n$ nodes and define $^S\Lmod_n=L_n+\epsilon \sum_{j\in S}e_{j} e_{j}^T$ for small $\epsilon>0$ and $e_{j}$ is $j^{th}$ column of $I_n$, the $n\times n$ identity matrix.\\
	Then the smallest eigenvalue of the perturbed Laplacian matrix up to the first order in $\epsilon$ is independent of:
	\begin{itemize}\vspace{-1mm}
		\item the number of edges in the graph $G_n$.
		\item the port indices (i.e. independent of set $S$) and its value is equal to $\cfrac{k\epsilon}{n}$. 
	\end{itemize}	
\end{theorem}
\noindent See the appendix for the proof.

\subsection{Scaling with respect to number ports and number of nodes} 
We saw in the previous subsection that the linear approximation is independent of the port index for an undirected, unweighted, general simple graph (i.e. path graph, general tree and also for graphs with cycles). We now study the $2^{nd}$ order term, which plays the key role about optimality of port indices. In this subsection, we first observe through an example in Figure~\ref{fig:IncreasingReturn} that a certain convexity exists between $\lambda_{\min}$ on number of ports (at port optimality). More precisely, from Figure~\ref{fig:IncreasingReturn}, we see that $\lambda_{\min}^*(k)$ at optimal $k$-ports is greater than $k$ times $\lambda_{\min}^*(1)$ at optimal $1$-port. 

For the special case of $k=1$ and $k=2$, we use Theorem~\ref{thm:1port_path_Eigenshift} and Theorem~\ref{thm:2ports_path_Eigenshift} to formulate the theorem below. To avoid divisibility constraints/assumptions, we assume $n$ to be large. \vspace{-2mm}
\begin{theorem}\label{thm:convexity}
	Consider a path graph of sufficiently large order $n$. Let $\lambda_{\min}(k)$ be the optimal shift in smallest eigenvalue as per Defn~\ref{def:optim:all3}[i](MP\textbf{L}SE) for $k$-ports selection ($k<n$).\\
	Then: $\lambda_{\min}(2)>2\times \lambda_{\min}(1)$ for up to $2^{nd}$ order approximation in $\epsilon$.
\end{theorem}
See the appendix for the proof.

\begin{figure}[h!]
	\centering
	\includegraphics[width=0.7\textwidth, height=.4\textwidth]{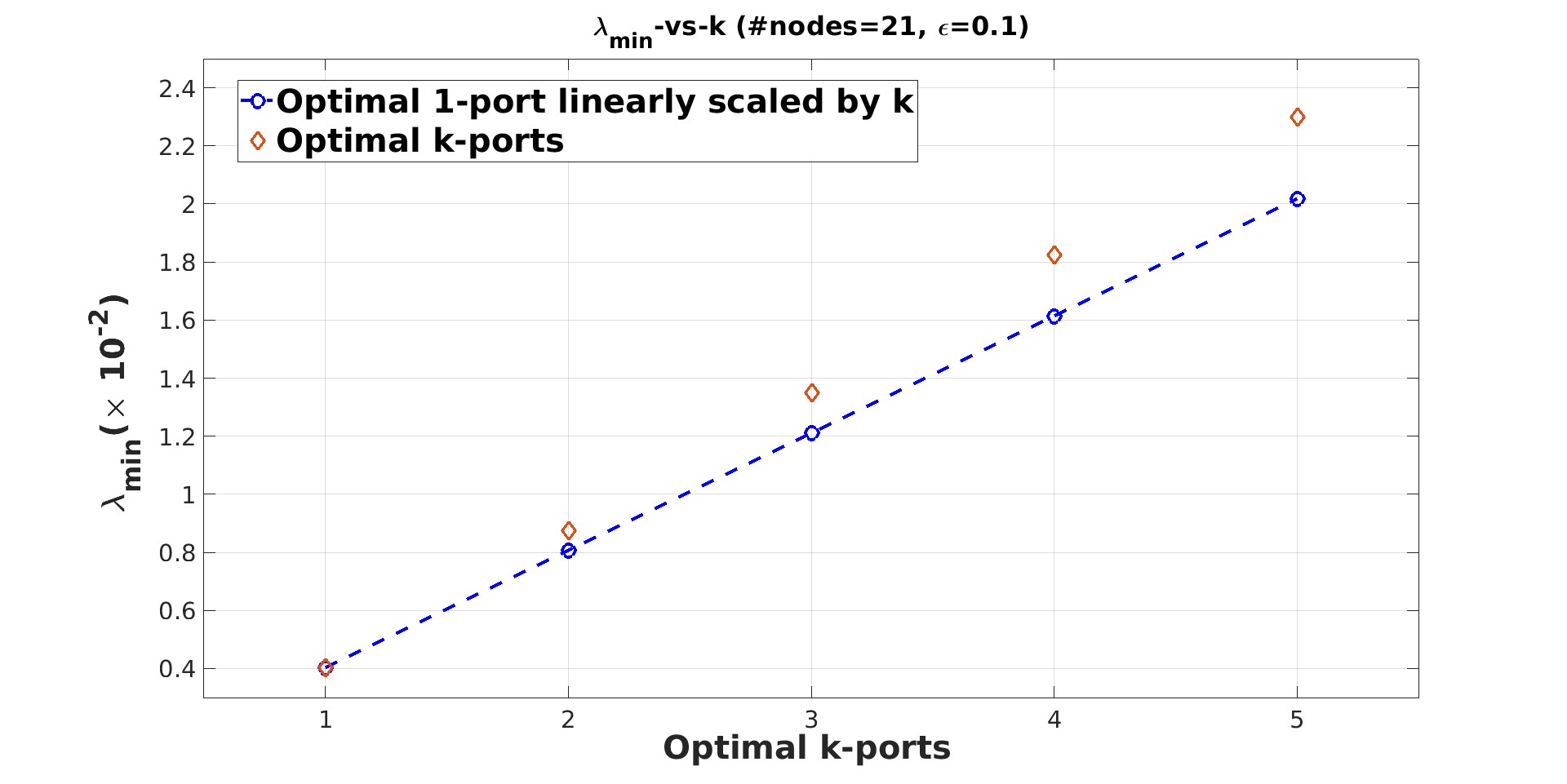}
	\caption{\small Path graph with multiple (optimal) ports (say $k$): shift achieved is \underline{more} than $k$ times the shift due to a single port}
\end{figure}\label{fig:IncreasingReturn}

\begin{remark}\label{rem:magnifyingReturn}
	Figure~\ref{fig:IncreasingReturn} shows the shift in $\lambda_{\min}$ with respect to the $k$-optimal ports, and it compares this shift with that of one optimal port linearly scaled to $k$ ports: the optimal $k$-port perturbation yields a greater shift than $k$ times $1$-port perturbation. Thus, $\lambda_{\min}(k)$ is a convex function of $k$. Thus provides more return than the scaled investment (in contrast with the law of diminishing returns). In a circuit analogy interpretation, $k$-port perturbation represents $k$ different passages for discharge, enabling the capacitor to discharge from the \underline{nearest} leakage resistance. Consequently, discharge through the $k$-optimal ports are faster compared to a single passage scaled by $k$. Thus, an RC circuit helps explain the convexity. Further, the circuit interpretation also explains a plausible reason for why we have a case opposite of `law of diminishing returns'.
\end{remark}

\subsection{Relation between $1$-optimal port and $2$-optimal ports: path graphs}

\noindent In the path graph network, $P_n$ and $P_{2n}$, we investigate the question: if we select $1$-optimal port and $2$-optimal ports respectively, what is the relationship between the smallest eigenvalue shifts that occur in both cases? From a circuit's viewpoint connecting two circuit in ``cascade" (loosely speaking), the question is will the larger circuit discharge faster or slower? The following theorem states relation between the smallest eigenvalue of the perturbed Laplacian matrix for one port selection and two ports selection (for a $2n$ node path graph). \vspace{-2mm}
\begin{theorem}\label{thm:LnL2nfirstEigenvalue}
	Let $n$ be odd, consider an $n$-nodes path graph Laplacian $L_n$. Suppose $L_n$ is perturbed optimally for a 1-port selection to get: $\Lmod_n:=L_n+\epsilon e_{p^*}e_{p^*}^T$, with $p^*=\cfrac{(n+1)}{2}$.
	
	\noindent Suppose a $2n$-nodes path graph Laplacian $L_{2n}$ is perturbed optimally for 2-ports selection to get: \\
	$\Lmodd_{2n}:=L_{2n}+\epsilon \{e_{p^*_1}e_{p^*_1}^T+e_{p^*_2}e_{p^*_2}^T\}$, with $p_1^*=\cfrac{n+2}{4} $ ~and~$p_2^*=\cfrac{3n+2}{4}$. 
	
	\noindent Then $\lambda_{\min}(\Lmod_n)=\lambda_{\min}(\Lmodd_{2n})$. 	
\end{theorem}
See the appendix for the proof. Using the Cauchy interlacing theorem\footnote{The eigenvalues of a real symmetric matrix are interlaced (with non-strict inequality) with the eigenvalues of its principal sub-matrix.} \cite[Theorem~10.1.1]{Parlett:SymmetricEig} interlacing is expected; in the above theorem we state equality for the smallest eigenvalues of $\Lmod_n$ and $\Lmodd_{2n}$. The above result might not be too surprising from a circuit's viewpoint given the symmetry in the network. The result below also is true, though the proof is quite tedious and also due to paucity of space and due to lack of a clear motivation/significance, we include the statement below as only an `observation' and do not pursue a proof. This result below relates other eigenvalues of $\Lmod_{n}$ and $\Lmodd_{2n}$.
\begin{observation}
	Let $n$ be odd, consider an $n$-nodes path graph Laplacian $L_n$. Suppose $L_n$ is perturbed by arbitrary $\epsilon$ 
	for 1-port selection: $\Lmod_n:=L_n+\epsilon e_{p^*}e_{p^*}^T$, with $p^*=\cfrac{(n+1)}{2}$.
	The eigenvalues of the perturbed Laplacian matrix $\Lmod_n$ is
	$\lambda_1< \lambda_2< \dots< \lambda_n$.\\
	Further, suppose a $2n$-nodes path graph Laplacian $L_{2n}$ is perturbed by arbitrary $\epsilon$ for 2-ports selection: ($\Lmodd_{2n}:=L_{2n}+\epsilon \{e_{p^*_1}e_{p^*_1}^T+e_{p^*_2}e_{p^*_2}^T\}$, with $p_1^*=\cfrac{n+2}{4}$ ~and~$p_2^*=\cfrac{3n+2}{4}$). The eigenvalues
	of the perturbed Laplacian matrix $\Lmodd_{2n}$ is $\mu_1< \mu_2< \dots< \mu_{2n}$.\\
	Then the following
	hold: $\lambda_1=\mu_{1}<\mu_{2}<\lambda_2=\mu_{3}<\mu_{4} < \lambda_3=\mu_{5}< \dots< \lambda_n=\mu_{2n-1}<\mu_{2n}$. 
\end{observation}
The above result is obvious for $\epsilon=0$ (as per (ii) of Proposition~\ref{prop:pathEigvalEigvec} from $P_n$ to $P_{2n}$ the eigenvalues are interpolated): the significance is that the result is in-fact is true for $\epsilon \neq 0$ too.

The following conjecture points towards the insensitivity of the $\lambda_{\min}$ with respect to edge addition at the optimal port selection: this is a generalization of Theorem~\ref{thm:LnL2nfirstEigenvalue} as its accommodates for all pair combinations of nodes for edge addition.
\begin{conjecture}\label{conj:PnvsP2n}
       Let $n$ be odd, consider a path graph $P_n$ Laplacian matrix $L_n$. Suppose for 1-port selection $L_n$ is perturbed to $\Lmod_n:=L_n+\epsilon e_{p^*}e_{p^*}^T$, with $p^*=\frac{(n+1)}{2}$. Denote the optimal perturbed path graph $\Pp$. Suppose a graph is formed as $\Pp\cup \Pp$ and suppose one edge is added between any pair of nodes of the graph $\Pp\cup \Pp$ such that the resulting graph is connected:  denote the resulting connected graph on $2n$ nodes
as $\Pvar$. Then, $\lambda_{\min}(\Pp)=\lambda_{\min}(\Pp\cup \Pp)=\lambda_{\min}(\Pvar)$.
\end{conjecture}

\section{Auxiliary results}\label{sec:auxiliary:results}
This section deals with auxiliary results, using which we prove
the main results. 

\noindent The following lemma provides an expression of the smallest eigenvalue of the Laplacian matrix of the path graph after a rank-one perturbation at one of the diagonal entries.

\begin{lemma}\label{lem:1port_path_EigenShift}
	Let $n$ be odd and let $L_n$, the Laplacian matrix for the path graph of $P_n$. Suppose $L_n$ is perturbed to $L_n+\epsilon e_pe_p^T$ for sufficiently small $\epsilon$ with $e_p$, the $p^{th}$ column of $I_n$, the $n\times n$ identity matrix. \vspace{-2mm}
	\begin{itemize}
		\item[a)] Then, the smallest eigenvalue of the perturbed Laplacian matrix, $\Lmod_n:=L_n+ \epsilon e_pe_p^T$ is given by: 
		\begin{align*}
			\lambda_{min}(\Lmod_{n,\epsilon}, p)
			=\cfrac{\epsilon}{n}-\cfrac{\epsilon^2}{4n}\sum_{j=2}^{n}\cfrac{\cos^2\bigl(\pi (j-1)(p-0.5)/n\bigr)}{\sin^2\left(0.5\pi (j-1)/n \right)\left\{\sum_{i=1}^{n}\cos^2\left(\pi  (j-1)(i-0.5)/n\right)\right\}}+O(\epsilon^3)	
		\end{align*}
		\item[b)] Assume $p\in\mathbb{R}$ (instead of $p\in \mathbb{Z}$). Then $\lambda_{min}(\Lmod_{n,\epsilon}, p)$
		gets \underline{locally minimized} at $p=p^*=\frac{n+1}{2}$.\\
	\end{itemize}
\end{lemma}

\noindent See appendix for the proof. The following lemma provides an expression of the smallest eigenvalue of the Laplacian matrix of the path graph after perturbation at \underline{two} of its diagonal entries. 
\begin{lemma}\label{lem:2ports_path_EigenShift}
	Let $L_n$, the Laplacian matrix for the path graph of $n$ nodes with $\cfrac{n}{2}$ being odd be perturbed to $\Lmodd_n=L_n+ \epsilon e_{p_1}e_{p_1}^T+\epsilon e_{p_2}e_{p_2}^T$ for small $\epsilon>0$ and $e_{p_1}$, $e_{p_2}$ are the $p_1^{th}$ and $p_2^{th}$ columns of $I_n$, the $n\times n$ identity matrix respectively. Then the smallest eigenvalue of the perturbed Laplacian matrix, is given by:
		\hspace{-4mm}
		\begin{small}
		\begin{align*}
			\lambda_{min}(\Lmodd_{n,\epsilon}, p_1, p_2)&
			=\cfrac{2\epsilon}{n}-\cfrac{\epsilon^2}{4n}\sum_{j=2}^{n}\cfrac{\Bigl\{\cos\bigl(\pi (j-1)(p_1-0.5)/n\bigr)+\cos\bigl(\pi (j-1)(p_2-0.5)/n\bigr)\Bigr\}^2}{\sin^2\left(0.5\pi (j-1)/n \right)\left\{\sum_{i=1}^{n}\cos^2\left(\pi  (j-1)(i-0.5)/n\right)\right\}}+O(\epsilon^3)			
	\end{align*}
\end{small}\vspace{-4mm}
		
\end{lemma}

\noindent See appendix for the proof. The following lemma provides an expression of the smallest eigenvalue of the Laplacian matrix of the path graph after perturbation of $k$ diagonal entries of the Laplacian matrix.
\begin{lemma}\label{lem:kports_path_EigenShift}
	Let $S \subset \{1,2,\dots, n\}$ with $|S|=k<n$. Let $L_n$, the Laplacian matrix for the path graph of $n$ nodes be perturbed for $k$ nodes selection to $^S\Lmod_n:=L_n+\epsilon \sum_{j\in S}e_{p_j} e_{p_j}^T$ for small $\epsilon>0$ with $e_{p_j}$ is the $p_j^{th}$ column of $I_n$, the $n\times n$ identity matrix. 
	
	\noindent Then the smallest eigenvalue of the Laplacian matrix is given by:\vspace{-2mm}
	\begin{small}	
	\begin{align*}
		\lambda_{min}(^S\Lmod_n)=\cfrac{k\epsilon}{n}-\cfrac{\epsilon^2}{4n}\sum_{j=2}^{n}\cfrac{\Bigl\{\sum_{i=1}^{k}\cos\bigl(\pi (j-1)(p_i-0.5)/n\bigr)\Bigr\}^2}{\sin^2\left(0.5\pi (j-1)/n \right)\left\{\sum_{i=1}^{n}\cos^2\left(\pi  (j-1)(i-0.5)/n\right)\right\}}+O(\epsilon^3).
	\end{align*}
	\end{small}
\end{lemma}

See appendix for the proof. The following proposition (from \cite{EigExpStacks}) formulates an expression for the eigenvector of a symmetric matrix perturbed by another small symmetric matrix. 
\begin{proposition}\label{prop:eigenvecExpan:stackexchange}
	\cite{EigExpStacks} Let $A\in \mathbb{R}^{n\times n}$ be a symmetric matrix with $n$ distinct eigenvalues $\lambda_{j}$, and corresponding eigenvector $v_j$. Let $P\in \mathbb{R}^{n\times n}$ be a symmetric matrix. Consider $\tilde{A}:=A+\epsilon P$, a perturbation of $A$, where $\epsilon\in\R$ is a small. \\
	Then, the perturbed eigenvector $\hat{v}_j$ is:\vspace{-7
		mm}
	\begin{equation}\label{eq:perturbed:eigenvector}
		 \hat{v}_j:=v_j +\epsilon \sum_{k=1, k\neq j}^{n}\cfrac{v_j^T P v_k}{\lambda_{j}-\lambda_{k}}v_k +O(\epsilon^2).
	\end{equation}	
\end{proposition}

The following lemma deals with the dependence of the smallest eigenvalue of the perturbed Laplacian matrix of the path graph $P_n$ ($P_{2n}$) on the value of $n$; the perturbation being at first diagonal entry for $L_n$ (first and last diagonal entries for $L_{2n}$). It also shows the relation between the smallest eigenvalue of the perturbed Laplacian matrix of the path graph $P_{2n}$, perturbed at first and last diagonal entries and $P_{n}$, perturbed at first diagonal entry.
\begin{lemma}\label{lem:perturbedLvsN}
	Consider a path graph of $n$ nodes, and perturb its Laplacian matrix $L_n$ to $\Lmod_n := L_n + e_1e_1^T$ and a path graph of $2n$ nodes, and perturb its Laplacian matrix $L_{2n}$ to $\Lmodd_{2n} := L_{2n} + e_1e_1^T+ e_{2n}e_{2n}^T$. Consider dependence of $\lambda_{\min}(\Lmod_n)$ on $n$.\\
	Then, the following hold. \vspace{-5mm}
	\begin{enumerate}
		\item $\lambda_{\min}(L_n+e_1e_1^T)=\lambda_{\min}(L_{2n}+e_1e_1^T+e_{2n}e_{2n}^T)=2-2\cos\biggl(\cfrac{\pi}{2n+1}\biggr)$.
		\item The smallest eigenvalue of $\Lmod_n$ (as well as of $\Lmodd_{2n}$) are monotonically strictly decreasing with respect to $n$. 
			\end{enumerate} 
\end{lemma}

\noindent See appendix for the proof. The following lemma provides an approximation of the smallest root of sum of two polynomials one of which is scaled by $\epsilon$. The smallest root is represented in powers of $\epsilon$. 
\begin{lemma} \label{lem:perturb:epsi:num:den:series}
	Let $a(s)$ and $b(s)$ be two finite-degree polynomials with $a(s) = a_0 + a_1s + a_2 s^2+ \cdots +s^n$ and $b(s) = b_0 + b_1s + b_2 s^2  + \cdots+ b_{n-1}s^{n-1}$.
	Suppose $a_0 = 0$,  $a_1 \ne 0$, and $b_0 \ne 0$.\\
	Define the polynomial $p_\epsilon(s): = a(s) + \epsilon b(s)$. \\
	Consider the root $\lambda_{min}(0)$ at the origin for $\epsilon = 0$ and
	the dependence of $\lambda_{min}$ on $\epsilon$:
	expand $\lambda_{min}(\epsilon)$ about $\epsilon=0$ in a series in $\epsilon$ as:
	$\lambda_{min}(\epsilon) = \beta_1 \epsilon + \beta_2 \epsilon^2 + \beta_3 \epsilon^3 + \cdots$.\\
	Then,
	\begin{equation} \label{eq:taylor:inverse}
		\beta_1 = \frac{-b_0}{a_1} \quad \mbox{  and  }  \quad  \beta_2 = \frac{a_1 b_1 b_0 -a_2 b^2_0}{a_1^3}.
	\end{equation}
\end{lemma}

\noindent See appendix for the proof. The following lemma provides approximation of the smallest root of sum of three polynomials two of which are scaled by $\epsilon$ and $\epsilon^2$ respectively. The root is represented in powers of $\epsilon$. 
\begin{lemma} \label{lem:perturb:doubleepsi:num:den:series}
	Let polynomials $a(s)$,  $b(s)$ and $c(s)$ be such that $a(s) = a_0 + a_1s + a_2 s^2+ \cdots +s^n$, $b(s) = b_0 + b_1s + b_2 s^2  + \cdots+ b_{n-1}s^{n-1}$
  and $c(s) = c_0 + c_1s + c_2 s^2  + \cdots+ c_{n-2}s^{n-2}$.
	Suppose $a_0 = 0$,  $a_1 \ne 0$, $b_0 \ne 0$ and $c_0 \ne 0$.\\
	Define the polynomial $p_\epsilon(s): = a(s) + \epsilon b(s)+ \epsilon^2 c(s)$. \\
	Consider the root $\lambda_{min}(0)$ at the origin for $\epsilon = 0$ and
	the dependence of this root on $\epsilon$:
	expand $\lambda_{min}(\epsilon)$ about the origin in a series expansion:
	$\lambda_{min}(\epsilon) = \beta_1 \epsilon + \beta_2 \epsilon^2 + \beta_3 \epsilon^3 + \cdots$.\\
	Then,\vspace{-2mm}
	\begin{equation} \label{eq:taylor:doubleinverse}
		\beta_1 = \frac{-b_0}{a_1} \qquad \mbox{  and  }  \qquad  \beta_2 =\cfrac{(a_1b_1b_0-a_2b_0^2-a_1^2c_0)}{a_1^3}.
	\end{equation}	
\end{lemma}

\indent See appendix for the proof. The proposition (from \cite{AlwanSaidiLaplacianCh}) below formulates the characteristic equation of $-L_n$, where $L_n$ is the  Laplacian matrix of the path graph. 
\begin{proposition} \label{prp:characteristic:path:explicit}
	\cite[p.~645]{AlwanSaidiLaplacianCh} Let $n \geqslant 2$ and consider the undirected, unweighted path graph $P_n$ on $n$ nodes and its Laplacian matrix $L_n$.\\
	The characteristic polynomial\footnote{We consider $-L_n$ instead of $L_n$ merely to circumvent the book-keeping of the sign-changes in the coefficients.} of $-L_n$ is:
	\begin{equation} \label{eq:path:graph:characteristic}
		\begin{split}
			\det (sI + L_n) = n s + \frac{n(n^2-1)}{6} s^2 + \frac{n(n^2-1)(n^2-4)}{120} s^3 +
			 \cdots + 2(n-1) s^{n-1} + s^n.
	\end{split}
	\end{equation}
\end{proposition}

\noindent The proposition below provides a recursive characteristic polynomial relation for  generic tridiagonal matrix. 

\begin{proposition} \label{prop:characteristic:Tri}
	\cite[eq(9)]{KulkarniEtal} Let $Q\in \mathbb{R}^{m \times m}$ be a tridiagonal matrix of the following form : \\
	$Q_m=\begin{bmatrix}
		a_1 & c_2 & 0 & 0&\dots & 0\\
		b_2 & a_2 & c_3 & 0 &  \dots & 0\\
		0 & b_3 & a_3 & c_4 & 0\dots & 0\\
		\vdots & \vdots & \ddots & \ddots & \ddots &\vdots \\
		0 & \dots & 0 & b_{m-1} & a_{m-1} & c_{m} \\
		0 & \dots & \dots & 0 & b_m & a_m \\
	\end{bmatrix}$ where $a_i,~b_i,~c_i \in \mathbb{R}$\\ \vspace{2mm}
	\noindent Let $\psi_{m}(s)$ be the characteristic polynomial of $Q_m$ as above. Define $\psi_{0}=1,~\psi_{1}=(s-a_1)$. 
	\vspace{-2mm}
	\begin{equation} \label{eq:tridiagonal:characteristic}		
		\hspace*{-6mm}\mbox{Then, for } m\geqslant 2 \qquad \qquad \psi_{m}= (s-a_m)\psi_{m-1}-b_mc_m\psi_{m-2}.	\hspace*{5.5cm} 	
	\end{equation}
\end{proposition}

\noindent The proposition below is an adaptation of the tridiagonal matrix characteristic polynomial (Proposition~\ref{prop:characteristic:Tri}) to path graph for a single rank perturbation of the Laplacian matrix.
\begin{proposition} \label{prp:characteristic:path:perturb}
	Let $n \geqslant 2$ and consider the undirected, unweighted path graph $P_n$ on $n$ nodes and its Laplacian $L_n$. Define $\Lmod_n := L_n + \epsilon e_p e_p^T $ for some $p\in \{ 1,2, \dots, n\}$. Let $\psi_{m}(s)$ be characteristic polynomial of a tridiagonal matrix of size $m\in \mathbb{Z}$ (refer Proposition~\ref{prop:characteristic:Tri}). \\
	The characteristic polynomial of $\Lmod_{n}$ is \vspace{-2mm}
	\begin{equation} \label{eq:path:characteristic:perturb}
			\det(sI -\Lmod_n) = (s-2-\epsilon)\psi_{p-1}\psi_{n-p}-\psi_{p-2}\psi_{n-p}-\psi_{p-1}\psi_{n-p-1}	
			\vspace{-2mm}
	\end{equation} 
	where $\psi_{p}=(s-2)\psi_{p-1}-\psi_{p-2}$, for $2<p\leqslant n$ with $\psi_{0}=1$, $\psi_{1}=(s-1)$ and $\psi_{2}=(s-1)(s-2)$.
\end{proposition}

\indent The following proposition is an adaptation of the tridiagonal matrix characteristic polynomial (Proposition~\ref{prop:characteristic:Tri}) to the path graph Laplacian with a two rank perturbation. 
\begin{proposition} \label{prp:characteristic:path:perturb:2}
	Let $n \geqslant 6$ and consider the undirected, unweighted path graph $P_n$ on $n$ nodes and its Laplacian $L_n$. Define $\Lmodd_n := L_n + \epsilon e_{p_1} e_{p_1}^T +\epsilon e_{p_2} e_{p_2}^T$ for $p_1, ~ p_2~\in \{ 1,2, \dots, n\}$ with $1 \leqslant p_1 < n/2 <  p_2 \leqslant n$. Let $\psi_{m}(s)$ be the characteristic polynomial of the tridiagonal matrix of size $m\in \mathbb{Z}$ (refer Proposition~\ref{prop:characteristic:Tri}).\\
	The characteristic polynomial of $\Lmodd_{n}$ is \vspace{-2mm}
	\begin{equation} \label{eq:path:graph:characteristic}
		\begin{split}
			\det(sI -\Lmodd_n) &= (s-2-\epsilon)^2 \psi_{p_1-1}\psi_{p_2-p_1-1}\psi_{n-p_2}\\
			&-(s-2-\epsilon)\Bigl\{\psi_{p_1-1}\psi_{p_2-p_1-1}\psi_{n-p_2-1}+\psi_{p_1-1}\psi_{p_2-p_1-2}\psi_{n-p_2}+\psi_{p_1-2}\psi_{p_2-p_1-1}\psi_{n-p_2}\\
			&+\psi_{p_1-1}\psi_{p_2-p_1-2}\psi_{n-p_2}\Bigr\}\\
			&+\Bigl\{\psi_{p_1-2}\psi_{p_2-p_1-1}\psi_{n-p_2-1}+\psi_{p_1-1}\psi_{p_2-p_1-3}\psi_{n-p_2}+\psi_{p_1-1}\psi_{p_2-p_1-2}\psi_{n-p_2-1}\Bigr\}\\
		\end{split}
		\vspace{-2mm}		
	\end{equation}
	where $\psi_{p}= (s-2)\psi_{p-1}-\psi_{p-2}, \quad \mbox{for } 2<p\leqslant n~~\mbox{with}
	~\psi_{0}=1,~~ \psi_{1}=(s-1),~~ \psi_{2}=(s-1)(s-2)$.
\end{proposition}

\indent The following lemma obtains the characteristic equation of the perturbed Laplacian of the path graph after perturbation at any one of the diagonal location of the matrix with coefficient of $s$ in the perturbation.
\begin{lemma} \label{lem:perturb:epsi:num:den}
	Consider an odd integer $n \geqslant 3$ and let $L_n$ be the Laplacian matrix of the path graph $P_n$. Define $p^* := (n+1)/2$ and $\Lmod_n := L_n + \epsilon e_j e_j^T $ for $j\in \{ 1,2, \dots, n\}$.\\
    Also let $a(s) := \det(sI-L_n)$ and 
    $b_j(s) := -1 + s(\frac{(n-1)^2 + 2(n-1)}{4} + (j-p^*)^2 )$
	+ (terms in $s^2$ $\&$ higher).
	
	\noindent Then, the characteristic polynomial of $\Lmod_n$ is \vspace{-2mm}
	\begin{equation} \label{eq:path:graph:characteristic:perturbed}
		\det (sI - L_n +\epsilon e_j e_j^T ) = a(s) + \epsilon b_j(s).
	\end{equation}
\end{lemma}

\indent The proof of Lemma~\ref{lem:perturb:epsi:num:den} follows by a careful and straight-forward application of Proposition~\ref{prp:characteristic:path:perturb}, in which the tridiagonal structure is utilized and the terms corresponding to $\epsilon^0$ and $\epsilon^1$ are systematically collected (for degree in $s\leqslant2$). Since the proof is primarily book-keeping along above lines, we skip the proof.

The following lemma obtains the characteristic equation of the perturbed Laplacian of the path graph after perturbation at any two of the diagonal locations of the matrix. 
\begin{lemma} \label{lem:perturb:epsi:num:den:two}
	Consider an even integer $n \geqslant 6$ such that $n/2$ is odd
	and let $L_n$ be the Laplacian matrix for the path graph $P_n$.\\
	Let $j_1$ and $j_2$ satisfy $1 \leqslant j_1 < n/2 <  j_2 \leqslant n$ and define
	$p_1^* := \cfrac{n+2}{4}$ and $p_2^* := \cfrac{3n+2}{4}$.
	
    \noindent Define $a(s) := \det(sI-L_n)$,
    
    \noindent     $b_{j_1,j_2}(s) := 2 - (\frac{3n^2-4}{8} -\frac{n(j_1-j_2)}{2} + (j_1-p^*_1)^2 + (j_2-p^*_2)^2 )  s 
	+ \mbox{(terms in $s^2$ $\&$ higher)}$.
	
	\noindent Further, let
	$c_{j_1,j_2}(s) := (j_2-j_1) + \mbox{(terms in $s$ $\&$ higher)}$.
	
    \noindent Consider the perturbed Laplacian matrix 
	$\Lmodd_n := L_n + \epsilon e_{j_1} e_{j_1}^T + \epsilon e_{j_2} e_{j_2}^T $.\\
	Then, the characteristic polynomial of $\Lmodd_n$ is \vspace{-2mm}
	\begin{equation} \label{eq:path:graph:characteristic:2pert}
		\det(sI - \Lmodd_n) = a(s) + \epsilon b_{j_1,j_2}(s) + \epsilon^2 c_{j_1,j_2}(s).
	\end{equation}
	\end{lemma}\vspace{-2mm}
\indent Just like proof of Lemma~\ref{lem:perturb:epsi:num:den}, the proof of Lemma~\ref{lem:perturb:epsi:num:den:two} also follows by a careful and straight-forward application of Proposition~\ref{prp:characteristic:path:perturb:2}, in which the tridiagonal structure is utilized and the terms corresponding to $\epsilon^0$, $\epsilon^1$ and $\epsilon^2$ are systematically collected (for degree in $s\leqslant2$); we skip this proof too for the same reason as for Lemma~\ref{lem:perturb:epsi:num:den}.

The lemma below proves that the smallest eigenvalue shift after perturbation is maximized at  $p^*$ location (the physical center of a path graph).
\begin{lemma}\label{lem:perturb:shiftmax:one}
	Consider an odd integer $n \geqslant 3$ and let $L_n$ be the Laplacian matrix for the path graph $P_n$.
	Let $j$ with  $1 \leqslant j < n $ and define
	$p^* := \cfrac{n+1}{2}$ which is the physical center of the path graph.
	Consider the perturbed Laplacian matrix
	$\Lmod_n := L_n + \epsilon e_{j} e_{j}^T$.
	
	\noindent Then the smallest eigenvalue shift of $\Lmod_n$, up to $2^{nd}$ order in $\epsilon$, is \vspace{-2mm}
	\[\lambda_{min}(\epsilon,j)= \beta_1 \epsilon +\beta_2 \epsilon^2=\frac{\epsilon}{n} - \frac{\epsilon^2}{12n^2} \left\{(n^2-1)+12(j-p^*)^2\right\}. \]
	 Further, $\lambda_{min}(\epsilon,j)$ gets maximized globally at $j=p^*$. 
\end{lemma}
	
\indent See appendix for the proof. The lemma below shows that the smallest eigenvalue shift after perturbation is maximized at the physical centers of the two halves of a path graph.
\begin{lemma}\label{lem:perturb:shiftmax:two}
	Consider an even integer $n \geqslant 6$ such that $n/2$ is odd
	and let $L_n$ be the Laplacian matrix for the path graph $P_n$.\\
	Let $j_1$ and $j_2$ satisfy $1 \leqslant j_1 < n/2 <  j_2 \leqslant n$ and 
	define $p_1^* := \cfrac{n+2}{4}$ and $p_2^* := \cfrac{3n+2}{4}$, the physical centers of each half of the path graph.
	
	\noindent Consider the perturbed Laplacian matrix\\
	$\Lmodd_n := L_n + \epsilon e_{j_1} e_{j_1}^T + \epsilon e_{j_2} e_{j_2}^T $.
	
	\noindent Then the smallest eigenvalue shift of $\Lmodd_n$, up to $2^{nd}$ order in $\epsilon$, is  
	\begin{align*}
		\lambda_{min}(\epsilon,j_1, j_2) = \beta_1 \epsilon +\beta_2 \epsilon^2 =\cfrac{2\epsilon}{n}  -\frac{\epsilon^2}{12n^2}\left[(n^2-4)+24\left\{(j_1-p_1^*)^2+(j_2-p_2^*)^2\right\} \right].  
	\end{align*}  
	Further, $\lambda_{min}(\epsilon,j_1, j_2)$ gets maximized globally at $j_1=p_1^*$, $j_2=p^*_2$. 
\end{lemma}
\noindent See appendix for the proof. 
\begin{table}[!h]
	\centering
	\caption{$\lambda_{min}$(scaled) vs ports : $2$-ports selection in $P_{14}$}
	\scalebox{1}{
	$\begin{tabular}{||c||c|c|c|c|c|c|c||}
		\hline
		Port & \multirow{2}{*}{8} & \multirow{2}{*}{9}& \multirow{2}{*}{10}& \multirow{2}{*}{11}& \multirow{2}{*}{12}& \multirow{2}{*}{13}& \multirow{2}{*}{14}\\
		\# &  & & & & & & \\
		\hline \hline
		1	& -1.455& -1.450&  -1.447& \textbf{-1.446}& -1.447& -1.450& -1.455\\ \hline
		2	& -1.450& -1.445&  -1.442& \textbf{-1.441}& -1.442& -1.445& -1.450\\ \hline
		3	& -1.447& -1.442&  -1.439& \textbf{-1.438}& -1.439& -1.442& -1.447\\ \hline
		4	& \textbf{-1.446}& \textbf{-1.441}&  \textbf{-1.438}& \textbf{-1.437}& \textbf{-1.438}& \textbf{-1.441}& \textbf{-1.446}\\ \hline
		5	& -1.447& -1.442&  -1.439& \textbf{-1.438}& -1.439& -1.442& -1.447\\ \hline
		6	& -1.450& -1.445&  -1.442& \textbf{-1.441}& -1.442& -1.445& -1.450\\ \hline
		7	& -1.455& -1.450&  -1.447& \textbf{-1.446}& -1.447& -1.450& -1.455\\ 		
		\hline \hline
	\end{tabular}$}
	
\end{table}
\noindent The following remark discusses about the independence of one port selection Then other w.r.t. Defn.~\ref{def:optim:all3}[i](MP\textbf{L}SE).
\begin{remark}  \label{rem:RC:initial:voltage:reconciliation}
An interesting and significant consequence of the above Lemma~\ref{lem:2ports_path_EigenShift} is the 
\underline{independence} of the roles that $j_1$ and $j_2$ play in the contributions towards
the second order effect in the shift of the minimum eigenvalue. This independence
points out that choice of one node for leakage does not affect the optimality of the other: this
distinguishes the RC circuit leakage resistance placement problem from the
`ambulance location' problem well-studied in the facility location literature.
One possible explanation for the independence is that in an ambulance location problem,
in the case when two central locations are to be found, 
once one location is fixed, then the optimality of the other is (reasonably)
dependent on the location of the fixed-one. However, in the case of RC circuits,
using the Courant-Fischer interpretation of the minimum eigenvalue and the
corresponding eigenvector, we interpret that for a given node-pair choice, for leakage the
initial conditions of the  corresponding `slowest mode' are such that the
discharge is \underline{slowest}. Unlike the ambulance location problem, here, the initial charge vector (the so-called slowest discharge vector) gets re-decided based on a change in the node-pair choice. Thus a re-distribution of initial voltage is possible in this problem; this perhaps explains the possibility of independence of one port choice from another. This re-distribution of the initial voltages (at $t=0$) of the
capacitors  depending on the choice of the two nodes where leakage 
resistances are placed, is the aspect that distinguishes this notion of
$k$-centrality from that of the ambulance-location problem and perhaps allows the independence of the two optimal locations.
\end{remark}

\section{Comparison of proposed metrics with other \\control-theoretic measures: general graphs}\label{sec:heuristics}

Except Theorem~\ref{thm:firstorder:general}, most results in this paper pertain to path graphs. In this section, we investigate using randomly generated
graphs, both trees and graphs with cycles, the extent to which we have agreement amongst the various metrics this paper proposes. We also
pursue two more control-theoretic measures.  Of course, each of these measures have different significances and they are not expected to agree 100\%.
Tables~\ref{tab:comp2} have the percentages of agreement with the first of the three definitions listed in Definition~\ref{def:optim:all3}.
In addition to these three, the following two notions have control-theoretic relevance and are described below.

\newcommand{\ONES}{{\mathbf{1}}}
\newcommand{\Kmin}{K_{\rm min}}
\vspace{-2mm}
\begin{itemize}
 \item {\bf{ARE solution $\Kmin$}}: Consider the RC circuit again associated to a graph, and instead of introducing a 
  leakage resistance at one or more of the chosen ports, one can consider charging the circuit externally and in this matter, if $i$ is the current provided through
 the chosen ports, and $v$ is the vector of voltages across these ports respectively, then $v^T i$ is the total power provided to the circuit, and minimizing
 $\int_{-\infty}^0 v^T i dt$ over all trajectories that start from an initially fully discharged circuit (at $t=-\infty$) to the vector of all capacitor voltages equal
  to 1~V (i.e. vector $\ONES\in\Rn$),  at $t=0$, is of typical control-theoretic relevance,
 in particular passivity studies. When choosing $k$ ports, one can define those $k$ ports
 as central where this energy required to charge the circuit is minimized over all combinations of $k$ ports: this amounts to studying different 
 Algebraic Riccati Equations (ARE) for each of the choice of $k$ ports, and minimizing $\ONES ^T \Kmin \ONES$ over all such choices of $k$ ports; $\Kmin$ being
 the minimum ARE solution w.r.t.  the passivity supply rate for the RC circuit corresponding to the particular choice of $k$ ports.
 \item {\bf Observability Gramian $Q$}: Consider again the RC circuit associated to the graph. In this measure, we ask the question: how much
 energy can be extracted from the chosen ports, say $k$ of them, from a given initial condition voltage vector $v_0$? Those $k$ ports are called central
 for which the energy that can be extracted is maximum (over all choices of $k$ ports): further the maximization is done w.r.t. the initial condition voltage vector equal
 to $\ONES$: this vector corresponding to the so-called `consensus' vector. See \cite{Tsakalis}, for example, for more details.
\end{itemize}
 \begin{table}[!h]
 	\begin{center}
 		\caption{\label{tab:ComparativeStudy} Comparison of various metrics for general graphs, taking into account choices of up to three ports.} \label{tab:comp2}
 		$\begin{tabular}{|l|r|c|c|c|c|c|}
 			\hline 
 			\multirow{4}{*}{\textbf{Graphs}} & \multirow{4}{*}{\textbf{n}} & 
 			\multicolumn{5}{|c|}{\textbf{$\%$ match with MP\emph{L}SE } ($\epsilon=0.01$) for $100$ graphs}  \\  \cline{3-7}
 			&  & 
 			\multirow{2}{*}{\textbf{M\emph{Sup}LE}} & \multirow{2}{*}{\textbf{M\emph{Sub}LE}} & \multirow{2}{*}{\textbf{Eigenvector}} &
 			\multirow{2}{*}{\textbf{$\mathbf{1^{T}} K_{min}\mathbf{1}$}} & \multirow{2}{*}{\textbf{$\mathbf{1^{T}}Q \mathbf{1}$}} \\ 
 			&  &   &  &  &  &  \\ \hline \hline
 			Path & $11$ &  $100\%$ & $100\%$ & $100\%$ & $100\%$ & $100\%$ \\ \hline
 			Tree & $7$ &  $100\%$ & $97.67\%$ & $66\%$ & $97\%$ & $59\%$ \\ \hline 
 			Tree & $9$ &  $100\%$ & $91\%$ & $69.67\%$ & $99.67\%$ & $63\%$ \\ \hline     
 			General & $7$ &  $100\%$ & $78.15\%$ & $40.74\%$ & $90\%$ & $80.37\%$ \\ \hline
 			General & $9$ &  $100\%$ & $70.5\%$ & $33.34\%$ & $97.92\%$ & $90.63\%$ \\ \hline \hline
 		\end{tabular}$
 	\end{center}   
 	\vspace{5mm}
 \end{table}
{\bf Laplacian eigenvector:} Various results in the earlier section indicate (for path graphs) about how a certain eigenvector of the Laplacian matrix has
components zero exactly for those nodes which are $k$-central. In our opinion, this result is a generalization (albeit only for path graphs) of the
well-known \cite{BapatSukanta} fact that, for trees, the Fiedler vector components increase in magnitude as one moves along branches away from the
`characteristic block'. Thus, while the characteristic node (when one exists) is the center for trees, our results for path graphs indicate that
eigenvectors for larger eigenvalues also contain information about $k$ centers. More precisely, when identifying the central $k$-nodes for a graph on $n$ nodes,
then the eigenvector $v_{k+1}$, corresponding to eigenvalue $\lambda_{k+1}$ (assuming that this eigenvalue is not repeated), has magnitude smallest at
those $k$-nodes which are `central' in the graph. This observation agrees for path graphs, as formulated and proved in the theorems in the main section.
In Tables~\ref{tab:comp2}, we check the extent to which this agrees with other metrics for general graphs: both trees and graphs with cycles.
The column marked `Eigenvector' corresponds to picking those nodes which are smallest in magnitude in the eigenvector $v_{k+1}$ of the Laplacian matrix,
and checking agreement about this set with the $k$-centers identified by MP\emph{L}SE.


\section{Concluding remark and future directions}\label{sec:conclusion}
In this work, we introduced two new measures, Maximized Perturbed Laplacian Smallest Eigenvalue (MP\textbf{L}SE) and Minimized Super-stochastic Largest Eigenvalue (M\textbf{Sup}-LE), for selecting the best $k$ ports in a graph. We further proved that the results obtained above metrics  perfectly aligned with the results obtained using the existing metric, Minimized Sub-stochastic Largest Eigenvalue (M\textbf{Sub}-LE) for the path graph.

We proved that the smallest eigenvalue of the Laplacian matrix for a path graph $P_n$ is equal to the smallest eigenvalue of a path graph $P_{2n}$ at optimal selection of one and two ports respectively. Additionally, we observed that the eigenvalues of the Laplacian matrix $\Lmod_n$ for the n-node path graph are all present as eigenvalues in the Laplacian matrix $\Lmodd_{2n}$ of the 2n-node path graph. These eigenvalues are interlaced with the other $n$ eigenvalues of the matrix. Furthermore, it has been showed that the smallest eigenvalue of the perturbed Laplacian matrix of the general graph, considering an up-to-first-order approximation in $\epsilon$, remains independent of the port indices.

When selecting two ports, we showed that the choice of one optimal port is independent of the other in the path graph. In observation, the choice of optimal port w.r.t. metrics (MP\textbf{L}SE and M\textbf{Sup}-LE) being independent is valid for any $k$ ports selection. Further, it is also observed that $k^{th}$ port selection remains independent of $k-1$ ports selection i.e. fixing one and searching for others optimal ports doesn't change the optimal location of other ports.

We have proved that the shift in the minimum eigenvalue ($\lambda_{\min}$) with respect to the $2$-optimal ports exceeds the shift obtained by linearly scaling a single optimal port to $2$ ports. Furthermore, our observation indicates that an optimal $k$-ports perturbation results in a larger shift than $k$ times the shift caused by a $1$-port perturbation. Consequently, the function $\lambda_{\min}(k)$ is convex and provides a greater return than the scaled investment (in contrast to the law of diminishing returns).

For future directions, we aim to demonstrate the independence of selecting one optimal port from another when choosing $k$ ports in the path graph. We aim to prove $\lambda_{\min}(k)$  convexity for $k$-ports selection. We plan to globally prove the selection of $k$ optimal ports for the path graph using all three metrics (MP\textbf{L}SE, M\textbf{Sub}-LE, and M\textbf{Sup}-LE). Our future work will focus on globally proving the selection of $k$ optimal ports for general graphs using our two metrics (MP\textbf{L}SE and M\textbf{Sup}-LE).

\bibliographystyle{IEEEtran}

\newpage
\section{Appendix}
\begin{proof}[Proof of Theorem \ref{thm:1port_path_Eigenshift}:] 
%
%
	\noindent \textbf{Part a)}
	\noindent With the setup given, Lemma~\ref{lem:1port_path_EigenShift} states eigenvalue shift function as follows: \vspace{-3mm}
	\begin{equation*}
		\lambda_{min}(p) =\cfrac{\epsilon}{n}-\cfrac{\epsilon^2}{4n}\sum_{j=2}^{n}\cfrac{\cos^2\Biggl(\cfrac{\pi (j-1) p}{n}-\cfrac{\pi (j-1)}{2n}\Biggr)}{\sin^2\left(0.5\pi (j-1)/n \right)\left\{\sum_{p=1}^{n}\cos^2\left(\frac{\pi p (j-1)}{n}-\frac{\pi(j-1)}{2n}\right)\right\}}+O(\epsilon^3)~~\mbox{terms}\\ 
	\end{equation*}
	Besides Lemma~\ref{lem:1port_path_EigenShift}, the eigenvalue shift function from Lemma~\ref{lem:perturb:shiftmax:one}:\\
	$\lambda_{min}(p)=\cfrac{\epsilon}{n} - \epsilon^2 \left[\cfrac{(n^2-1)+12(p^*-p)^2}{12n^2}\right] $
	
	\noindent In addition to above result Lemma~\ref{lem:perturb:shiftmax:one} also showed that at $p~=~p^*$, $\lambda_{min}(p)$ gets maximized. 
	
	\noindent Hence, combining the values of above two expressions at $p=p^*$ gives: \vspace{-5mm}
	\begin{align*}
		\underset{p}{\max}~\lambda_{min}=\lambda_{min}|_{p=p^*}=&\cfrac{\epsilon}{n}-\cfrac{\epsilon^2}{4n}\sum_{j=2}^{n}\cfrac{\cos^2\left(\pi (j-1)/2\right)}{\sin^2\left(0.5\pi (j-1)/n \right)\left\{\sum_{p=1}^{n}\cos^2\left(\frac{\pi p (j-1)}{n}-\frac{\pi(j-1)}{2n}\right)\right\}}\\
		=&\frac{\epsilon}{n} - \epsilon^2 \left[\frac{(n^2-1)}{12n^2}\right]
	\end{align*}
	This proves part (a) of the theorem.\\
	
	\noindent \textbf{Part~b)} 
	\noindent Theorem~\ref{thm:kports_path}, the Defn~\ref{def:optim:all3}[ii](M\textbf{Sub}-LE) for $k$-ports selection states that $p^*_i=\cfrac{(2i-1)n+k}{2k}$ with $i=1,2, \dots, k$ are optimal $k$ ports under the divisibility assumption.
	
	\noindent Thus, for $k=1$, $p=p^*=\frac{n+1}{2}$ which is same as Defn~\ref{def:optim:all3}[i](MP\textbf{L}SE) as mentioned in part (a) referring Lemma~\ref{lem:perturb:shiftmax:one}. 
	
	\noindent Further, Theorem~\ref{thm:superstochastic_eigenshift} shows the equivalence between the definitions MP\textbf{L}SE and M\textbf{Sup}-LE.
	
	\noindent Hence, if $p=p^*$ is the best 1-center w.r.t. MP\textbf{L}SE, it is center w.r.t. M\textbf{Sup}-LE also.
	
	\noindent Thus, $p=p^*$, is the center w.r.t. all three Defn~\ref{def:optim:all3}.\\
	This completes, the proof of part (b) of the theorem.\\

	\noindent \textbf{Part~c)} 
%
	\noindent Consider the following general expression for the eigenvectors of the path graph from Proposition~\ref{prop:pathEigvalEigvec}. 
\begin{equation*}
	v_q(p)=\cos\Big(\frac{\pi (q-1) p}{n}-\frac{\pi (q-1)}{2n}\Big)\quad \mbox{with}~~q=\{1,2, \dots,n\}.
\end{equation*}
Choosing the Fiedler vector 
$v_2(p)$, for $1$-port selection, from Proposition~\ref{prop:pathEigvalEigvec}, we see that part (c) follows and thus  
%
%
completes the proof of Theorem~\ref{thm:1port_path_Eigenshift}.
\end{proof} 

\begin{proof}[Proof of Theorem \ref{thm:2ports_path_Eigenshift}:]
%
%
	\noindent \textbf{Part~a)} 
With the given setup, Lemma~\ref{lem:2ports_path_EigenShift} states the eigenvalue shift satisfies:\\
	\scalebox{1}{
$	\lambda_{min}(\Lmodd,p_1, p_2, \epsilon)=\cfrac{2\epsilon}{n}-\cfrac{\epsilon^2}{4n}\sum_{j=2}^{n}\cfrac{\Biggl\{\cos\biggl(\cfrac{\pi (j-1) p_1}{n}-\cfrac{\pi (j-1) }{2n}\biggr)+\cos\biggl(\cfrac{\pi (j-1) p_2}{n}-\cfrac{\pi (j-1) }{2n}\biggr)\Biggr\}^2}{\sin^2\left(0.5\pi (j-1)/n \right)\left\{\sum_{p=1}^{n}\cos^2\left(\frac{\pi p (j-1)}{n}-\frac{\pi(j-1)}{2n}\right)\right\}}$}

\noindent Besides Lemma~\ref{lem:2ports_path_EigenShift}, the eigenvalue shift function from Lemma~\ref{lem:perturb:shiftmax:two}:\\
	$\lambda_{min}(\Lmodd_{n,\epsilon}, p_1, p_2) = \cfrac{2\epsilon}{n}  -\epsilon^2 \left[\cfrac{n^2-4}{24n}+\{\cfrac{(p_1-p_1^*)^2+(p_2-p_2^*)^2}{n}\}\right]$\\
	\noindent In addition to above result Lemma~\ref{lem:perturb:shiftmax:two} also showed that (at $p_1~=~p_1^*=\frac{n+2}{4}$, $p_2=p_2^*=\frac{3n+2}{4}$)
	evaluating the expression at $p_1^*$ and $p_2^*$,  $\lambda_1(p_1, p_2)$ gets maximized over all pairs $p_1,~p_2 \in \Z$ at $p_1^*$ and $p_2^*$. 
	
	\noindent Hence, the value of above two expressions at $p_1=p_1^*=\frac{n+2}{4}$, $p_2=p_2^*=\frac{3n+2}{4}$ is: \vspace{-3mm}
	
	\begin{equation*}
		\begin{split}
			\underset{p_1,p_2}{\max}~ \lambda_{min}(\Lmodd_{n,\epsilon}, p_1, p_2)&=\lambda_{min}(\Lmodd_{n,\epsilon}, p_1, p_2)\Bigg |_{p_1=p_1^*, p_2=p_2^*}=\frac{2\epsilon}{n} -\epsilon^2\left[\frac{n^2-4}{12n^2} \right]\\
			&=\cfrac{2\epsilon}{n}-\cfrac{\epsilon^2}{4n}\sum_{j=2}^{n}\cfrac{\Bigl\{\cos\left(\pi (j-1)/4 \right)+\cos\left(3\pi (j-1)/4 \right)\Bigr\}^2}{\sin^2\left(0.5\pi (j-1)/n \right)\left\{\sum_{p=1}^{n}\cos^2\left(\frac{\pi p (j-1)}{n}-\frac{\pi(j-1)}{2n}\right)\right\}}
		\end{split}
	\end{equation*}
	This proves the part (a) of the theorem.
	
	\noindent \textbf{Part~b)} 
	\noindent Theorem~\ref{thm:kports_path}, the Defn~\ref{def:optim:all3}[ii](M\textbf{Sub}-LE) for $k$-ports selection states $p^*_i=\cfrac{(2i-1)n+k}{2k}$ with $i=1,2, \dots, k$.
	
	\noindent Henceforth, for $k=2$, $p_1=p_1^*=\frac{n+2}{4}$, $p_2=p_2^*=\frac{3n+2}{4}$ which is same as Defn~\ref{def:optim:all3}[i](MP\textbf{L}SE) as shown in part (a). 
	
	\noindent Further, Theorem~\ref{thm:superstochastic_eigenshift} shows the equivalence between the Definitions~\ref{def:optim:all3}[i](MP\textbf{L}SE) and \ref{def:optim:all3}[iii](M\textbf{Sup}-LE). Hence, if $p_1=p_1^*$, $p_2=p_2^*$ are the two centers w.r.t. Defn~\ref{def:optim:all3}[i](MP\textbf{L}SE), then they are the two centers w.r.t. Defn~\ref{def:optim:all3}[iii](M\textbf{Sup}-LE) also.
	
	\noindent This proves, $p_1=p_1^*=\frac{n+2}{4}$, $p_2=p_2^*=\frac{3n+2}{4}$, are the center w.r.t. all three Defn~\ref{def:optim:all3} and completes, the proof of part (b) of the theorem.
	
	\noindent \textbf{Part~c)} 
	\noindent From Proposition~\ref{prop:pathEigvalEigvec}, the eigenvectors of the path graph. 
	\begin{equation*}
		v_q(p)=\cos\Big(\frac{\pi (q-1) p}{n}-\frac{\pi (q-1)}{2n}\Big)\quad \mbox{with}~~q=\{1,2, \dots,n\}
	\end{equation*}
		Choosing $q$=3, eigenvector $v_3 \in \R^n$, for $2$-ports selection, from Proposition~\ref{prop:pathEigvalEigvec}, and equating $\cfrac{\pi (q-1) p}{n}-\cfrac{\pi (q-1)}{2n}=\cfrac{\pi}{2}$ and $\cfrac{3\pi}{2}$, we get part (c) and this completes
%
%
%
 the proof of part (c) and thus completes the proof of Theorem~\ref{thm:2ports_path_Eigenshift}.
\end{proof} 

\begin{proof}[Proof of Theorem~\ref{thm:kports_path}:]	
%
%
%
	\noindent \textbf{Part~a)} 
	This part of the proof consists of two parts: 
	\begin{itemize}
		\item that $p_i^*$ forming the set $S^*$ indeed achieves $\lambda_{\max}(^{S^*}\hat{Z})$
		\item any other $S\subset \bar{n}$, $|S|=k$ result in $\lambda_{\max}(^S\hat{Z})>\lambda_{\max}(^{S^*}\hat{Z})$
	\end{itemize}
	
	We know that \\
	$\arg \underset{\underset{\mbox{with $|S|=k$}}{\mbox{all subsets $S$ }}}{\max} \lambda_{\min}(^S\!\Lmod_n)=\arg \underset{\underset{\mbox{with $|S|=k$}}{\mbox{all subsets $S$ }}}{\min} \lambda_{\max}(-^S\!\Lmod_n)$ or $\arg \underset{\underset{\mbox{with $|S|=k$}}{\mbox{all subsets $S$ }}}{\min} \lambda_{\max}(-^S\!\Lmod_n \tau)$.
	
	\noindent Definition~\ref{def:optim:all3}(B) states that $k$ rows and $k$ columns are removed from $Z_n=I_n-\tau~^S\!\Lmod_n$ for searching for the $\underset{|S|=k}{\min}~\lambda_{\max}(^S\hat{Z})$
	
	\noindent However, after removing $k$ rows and $k$ columns of $L_n$:
	 $\hat{L}_n\equiv\hat{\Lmod}_n \equiv \hat{Z}$, 
	$\underset{|S|=k}{\min}~\lambda_{\max}(\hat{Z}) \iff \underset{|S|=k}{\max}~\lambda_{\min} (\hat{\Lmod}_n)$. 
	
	\noindent Further, removing $k$ rows and $k$ columns from $\Lmod_{n}$ and assuming consecutive rows/columns are not removed, we get that $\hat{\Lmod}_n$ is now a block diagonal matrix comprised of totally $k+1$ matrices: $k$-1 tridiagonal Toeplitz matrix $T^o$ and 2 tridiagonal pseudo-Toeplitz matrix $T^{po}$. 
	
	\noindent Lemma~\ref{lem:perturbedLvsN} showed that to maximize $\lambda_{\min}$, we get: $\lambda_{\min}(T_{2\ell}^o)=\lambda_{\min}(T^{po}_\ell)$ i.e. tridiagonal Toeplitz matrix $T^o$ and 2 tridiagonal pseudo-Toeplitz matrix $T^{po}$ of sizes $2\ell$  and $\ell$ respectively ($\ell\in\mathbb{Z}$).
	
	\noindent We know that after $k$ rows and $k$ columns removal, we get an $(n-k)\times(n-k)$ size matrix which is block diagonal. \\
	Therefore, $\ell+\underbrace{2\ell+2\ell+\dots+2\ell}_\text{$k-1$}+\ell=n-k$ $\implies$ $\ell=\cfrac{n-k}{2k}$.\\
	\noindent So, the optimal locations are at $p^*_1=\ell+1=\cfrac{n-k}{2k}+1=\cfrac{n+k}{2k}$.\\
	Similarly, $p^*_2=p^*_1+2\ell+1=\cfrac{n+k}{2k}+2\cfrac{n-k}{2k}+1=\cfrac{3n+k}{2k}$\\
	Hence, in general $p^*_i=\cfrac{(2i-1)n+k}{2k}$ with $i=1,2, \dots, k$. 
	
	\noindent At optimality, $^{S^*}\hat{Z}$ is tridiagonal Toeplitz matrix whose largest eigenvalue using formula from \cite[p.~116]{KouachiTridiag} for $a=c=1/2$ and $b=0$ for block of size $2\ell\times 2\ell$, we get $\lambda_{\max}(^{S^*}\hat{Z})=\cos(\pi/(2\ell+1))=\cos(k\pi/n)$.
	
	\noindent Since optimality provide only two sizes of the block diagonal sub-matrices, we first deduce that removal of $k$ rows and $k$ columns at \underline{any of the other locations} would lead to at least one of the following two cases: 
	\begin{itemize}\setlength{\itemindent}{10mm}
	\item [Case 1)]  At least one of the $k-1$ tridiagonal Toeplitz matrices $T^o$ is of size strictly smaller than $\frac{n-k}{k}$. This would result in $\lambda_{\min}(\Lmod)<\lambda_{\min}(\Lmod^*)$ (see Lemma~\ref{lem:perturbedLvsN}).
		\item [Case 2)] At least one of the $2$ tridiagonal pseudo-Toeplitz matrices $T^{po}$ is of size strictly smaller than $\frac{n-k}{2k}$. This would also result in $\lambda_{\min}(\Lmod)~<~\lambda_{\min}(\Lmod^*)$ (see Lemma~\ref{lem:perturbedLvsN}). Since any of the two cases would result in $\lambda_{\min}(\Lmod)~<~\lambda_{\min}(\Lmod^*)$, we note that the $2^{nd}$ part of the proof of part (a) is completed.
	\end{itemize}
	
	\noindent \textbf{Part~b)} 
	\noindent From Proposition~\ref{prop:pathEigvalEigvec}, the eigenvectors $v_q\in \Rn$ of the path graph for eigenvalue $\lambda_{q}$ 
	Choosing $q=k+1$, the eigenvector $v_{k+1}(p)$, for $k$-ports selection in Proposition~\ref{prop:pathEigvalEigvec}, we get \\
	$\cfrac{k\pi p}{n}-\cfrac{k\pi}{2n}=\cfrac{\pi(2i-1)}{2}$ for $i \in \mathbb{N}$ $\iff$ $(2p-1)k=n(2i-1)$ $\iff p=\cfrac{n(2i-1)+k}{2k}$
	
	\noindent (Note: $p's \in \mathbb{N}$,  $\frac{n}{k}$ is odd).
	
	\noindent So, $v_{k+1}(j)=0$ $\iff$ $j\in \{p_1^*, p_2^*, \dots p_k^*\}$. \\
	This completes, the proof of part (b) of the theorem. This completes the proof of Theorem~\ref{thm:kports_path}. 
	\end{proof}

\begin{proof}[Proof of Theorem~\ref{thm:superstochastic_eigenshift}:]
	The Laplacian matrix $L_n$ of path graph $P_n$ is perturbed as:\\
	$^S\!\tilde{L}_n:=L_n-\epsilon\sum_{j\in S}e_je_j^T$ where $S \subset \{1,2, \dots, n\},$ with $|S|=k<n$ and $\epsilon>0$.
	
	\noindent $^S\!\tilde{Z}_n:=Z_n+\epsilon\sum_{j\in S}e_je_j^T$
	where $Z_n=I_n-\tau L_n$, where $\tau<\Delta^{-1}$ with $\Delta$: the maximum degree of graph.
	
	\noindent Definition~\ref{def:optim:all3}[i](MP\textbf{L}SE) identifies nodes as center if at these $k$ nodes $\lambda_{\min}(\Lmod_n)$ gets maximized. 
	
	\noindent Defn.~\ref{def:optim:all3}[iii](M\textbf{Sup}-LE) identifies nodes as center if at these $k$ nodes $ \lambda_{\max}(\tilde{Z}_n=I_n-\tau \Lmod_n)$ get minimized.
	
%
	\noindent Now obtain stochastic matrix $Z_n$ corresponding to $L_n$ as: $Z_n=I_n-\tau L_n$. 
	
	\noindent $^S\!\tilde{Z}_n=I_n-^S\!\!\Lmod_n \tau =Z_n+\tau \epsilon \sum_{j\in S}e_je_j^T$. 
	
	\noindent Since scaling a matrix by $\tau>0$ only scales the eigenvalues, we first note \vspace{-2mm}
	\begin{equation}\label{eq:argmax:tau}
		\arg \underset{\underset{\mbox{with $|S|=k$}}{\mbox{all subsets $S$ }}}{\max} \lambda_{\min}(^S\!\Lmod_n)=\arg \underset{\underset{\mbox{with $|S|=k$}}{\mbox{all subsets $S$ }}}{\max} \lambda_{\min}(^S\!\Lmod_n \tau),
	\end{equation}
	and noting the negative sign, this equals
	\begin{equation}\label{eq:argmax:tau:1}
		\arg \underset{\underset{\mbox{with $|S|=k$}}{\mbox{all subsets $S$ }}}{\min} \lambda_{\max}(-^S\!\Lmod_n \tau ). 
	\end{equation}
	In addition, we observe that the operation $P\rightarrow P+cI$ merely shifts the eigenvalues accordingly; this gives
	\vspace{-2mm}
	\begin{equation}\label{eq:argmax:tau:2}
		\arg \underset{\underset{\mbox{with $|S|=k$}}{\mbox{all subsets $S$ }}}{\min} \lambda_{\max}(-\tau~^S\!\Lmod_n)=\arg \underset{\underset{\mbox{with $|S|=k$}}{\mbox{all subsets $S$ }}}{\min} \lambda_{\max}(I_n- ^S\!\!\Lmod_n \tau). 
	\end{equation}
	Hence, from equations (\ref{eq:argmax:tau}-\ref{eq:argmax:tau:2}): \\
	$\arg \underset{\underset{\mbox{with $|S|=k$}}{\mbox{all subsets $S$ }}}{\max} \lambda_{\min}(^S\!\Lmod_n)~=~\arg \underset{\underset{\mbox{with $|S|=k$}}{\mbox{all subsets $S$ }}}{\min} \lambda_{\max}(I_n~-~^S\!\Lmod_n \tau)$.
	
	\noindent Thus, Defn.~\ref{def:optim:all3}[i](MP\textbf{L}SE) and [ii](M\textbf{Sup}-LE) result in the same $k$ nodes being identified as the $k$-centers. \\
	This completes the proof of Theorem~\ref{thm:superstochastic_eigenshift}.
\end{proof}

\begin{proof}[Proof of Theorem~\ref{thm:firstorder:general}:]
	Consider Laplacian is perturbed at any of the $k$-diagonal entry at $\epsilon$, $|S|=k$:\\ $^S\!\Lmod_n=L_n+\epsilon\sum_{i\in S}e_ie_i^T=L_n+\epsilon V$.\\
	Let perturbed smallest eigenvalue (upto first order approx in $\epsilon$) be $\hat{\lambda}_{1}=\lambda_{1}+\epsilon \beta_{1}$\\
	and using Proposition~\ref{prop:eigenvecExpan:stackexchange}, the corresponding eigenvector (upto first order approx in $\epsilon$) be, \\
	$\hat{v}_1(\epsilon)=v_1+\epsilon \sum_{j=2, }^{n}\cfrac{v_1^T V v_j}{\lambda_{1}-\lambda_{j}}v_j$.\\
	For ease of manipulation, define the coefficient of $\epsilon$ as $\Cchi:=\sum_{j=2, }^{n}\cfrac{v_1^T V v_j}{\lambda_{1}-\lambda_{j}}v_i$ with $\Cchi \in \Rn$ $\implies \hat{v}_1 =v_1+ \epsilon \Cchi$. 
	
	\noindent Thus, for perturbed Laplacian matrix eigenvalue, eigenvector equation, we have: \vspace{-2mm}
	\begin{center}
		$(L_n+\epsilon V)(v_1+\epsilon \Cchi)=(\lambda_1+\epsilon \beta_{1})(v_1+\epsilon \Cchi).$
	\end{center}\vspace{-2mm}
	Upon simplifying the above equation, we get\vspace{-2mm}
	\begin{center}
		$L_nv_1+\epsilon L_n \Cchi+\epsilon V v_1+\epsilon^2 V \Cchi=\lambda_1 v_1+\epsilon\lambda_1\Cchi+\epsilon\beta_{1} v_1+\epsilon^2\beta_{1}\Cchi$.
	\end{center}
	On equating the terms of $1^{st}$ order in $\epsilon$, we get\vspace{-4mm} 
	\begin{center}
		$L_n \Cchi+V v_1=\lambda_1\Cchi+\beta_{1} v_1$,\\
		$L_n \sum_{j=2, }^{n}\cfrac{v_1^T V v_j}{\lambda_{1}-\lambda_{j}}v_j+\sum_{i\in S}e_ie_i^T v_1=\lambda_1\sum_{j=2, }^{n}\cfrac{v_1^T V v_j}{\lambda_{1}-\lambda_{j}}v_j+\beta_{1} v_1$.
	\end{center}
	On pre-multiplying both sides by $v_1^T$ in above equation and using orthogonality of eigenvectors of symmetric matrix, $v_1^T v_{j}=0$ for $j=2,3,\dots,n$, we get:  \vspace{-4mm}
	\begin{center}
		$\beta_{1} v_1^Tv_1=v_1^T\sum_{i\in S}e_ie_i^T v_1$ $\implies \beta_{1} =\cfrac{v_1^T\sum_{i\in S}e_ie_i^T v_1}{ v_1^Tv_1}=\cfrac{k}{n}$
	\end{center}\vspace{-2mm}
	Thus the first order approx of $\hat{\lambda}_1=\lambda_{1}+\epsilon \cfrac{k}{n}$ is independent of port indices and also independent of the number of edges in the graph $G_n$. 
\end{proof}

\begin{proof}[Proof of Theorem~\ref{thm:convexity}:]
	For $1$ optimal port selection, for $n$ nodes from Theorem~\ref{thm:1port_path_Eigenshift}, we have
	\begin{equation}\label{eq:opt1}
		\lambda^*_{min}(1)=\frac{\epsilon}{n} - \epsilon^2 \left[\frac{n^2-1}{12n^2}\right].
	\end{equation}
	For $2$ optimal ports selection for $n$ nodes from Theorem~\ref{thm:2ports_path_Eigenshift}, we have
	\begin{equation}\label{eq:opt2}
			\lambda^*_{min}(2)=\frac{2\epsilon}{n} -\epsilon^2\left[\frac{n^2-4}{12n^2} \right].
	\end{equation}
	On multiplying equation~\eqref{eq:opt1} by $2$ for linear scaling and subtracting from equation~\eqref{eq:opt2}, we get: \\
	$\lambda^*_{min}(2)-2\lambda^*_{min}(1)=\cfrac{2\epsilon}{n} -\epsilon^2\left[\cfrac{n^2-4}{12n^2} \right]-\cfrac{2\epsilon}{n} + \epsilon^2 \left[\cfrac{(n^2-1)}{6n^2}\right]=\epsilon^2\cfrac{(n^2+2)}{12n^2}>0$.\\
	Thus, $\lambda^*_{min}(2)>2\lambda^*_{min}(1)$ by considering upto $2^{nd}$ order terms in $\epsilon$. 
\end{proof}

\begin{proof}[Proof of Theorem~\ref{thm:LnL2nfirstEigenvalue}:]
	Consider a path graph $P_n$, Laplacian $L_n$ with $n$ odd is perturbed optimally for 1-port selection:\\ $\Lmod_n:=L_n+\epsilon e_{p^*}e_{p^*}^T$, with $p^*=\cfrac{(n+1)}{2}$.
	
	\noindent Further, path graph $P_{2n}$, Laplacian $L_{2n}$ with $n$ odd is perturbed optimally for  2-ports selection: \\$\Lmodd_{2n}:=L_{2n}+\epsilon \{e_{p^*_1}e_{p^*_1}^T+e_{p^*_2}e_{p^*_2}^T\}$, with $p_1^*=\cfrac{n+2}{4}$ ~and~$p_2^*=\cfrac{3n+2}{4}$. 
	
	
	\noindent From Theorem \ref{thm:1port_path_Eigenshift}, we have: 
	$\lambda_{min}(\Lmod_{n,\epsilon}, p)\Bigg |_{p=p^*}=\lambda^*_{min}(\Lmod_n)=\cfrac{\epsilon}{n}  - \epsilon^2\left[\cfrac{n^2-1}{12n^2}\right]$.\\	
	Further from Theorem \ref{thm:2ports_path_Eigenshift}, we have: 
	$\lambda_{min}(\Lmodd_{n,\epsilon}, p_1, p_2)\Bigg |_{p_1=p_1^*, p_2=p_2^*}=\lambda^*_{min}(\Lmodd_n)=\cfrac{2\epsilon}{n}   -\epsilon^2\left[\cfrac{n^2-4}{12n^2} \right]$.\\
	Thus, for size $n$ in $\Lmodd_{2n}$, replacing $n$ by $2n$, we get:\\ $\lambda^*_{min}(\Lmodd_{2n})=\cfrac{2\epsilon}{2n}   -\epsilon^2\left[\cfrac{4n^2-4}{12\times 4n^2} \right]=
	\cfrac{\epsilon}{n}   -\epsilon^2\left[\cfrac{n^2-1}{12 n^2} \right]=\lambda^*_{min}(\Lmod_n)$
	
	\noindent This completes the proof of Theorem~\ref{thm:LnL2nfirstEigenvalue}.
\end{proof}

\begin{proof}[Proof of Lemma~\ref{lem:1port_path_EigenShift}:]
	Suppose Laplacian matrix $L_n$ of path graph $P_n$ is perturbed to:\\ $\Lmod_{n}=L_n+\epsilon e_je_j^T$ where $j\in \{1,2,3, \dots n\}$ and small $|\epsilon|$.
	
	\noindent \textbf{Part~a}) 
	\noindent From Lemma~\ref{lem:kports_path_EigenShift}, for $k-$ports selection, we have: \vspace{-4mm} \begin{small}\[\lambda_{min}(^S\!\Lmod_n)=\cfrac{k\epsilon}{n}-\cfrac{\epsilon^2}{4n}\sum_{j=2}^{n}\cfrac{\Biggl\{\sum_{i=1}^{k}\cos\biggl(\cfrac{\pi (j-1) p_i}{n}-\cfrac{\pi (j-1) }{2n}\biggr)\Biggr\}^2}{\sin^2\left(0.5\pi (j-1)/n \right)\left\{\sum_{p=1}^{n}\cos^2\left(\frac{\pi p (j-1)}{n}-\frac{\pi(j-1)}{2n}\right)\right\}}+O(\epsilon^3)\] 
	\end{small}
	\noindent For $1-$port selection, using $k=1$ in above equation, we get: \\
	\noindent $\lambda_{min}(\Lmod_{n,\epsilon}, p)=\cfrac{\epsilon}{n}-\cfrac{\epsilon^2}{4n}\sum_{j=2}^{n}\cfrac{\cos^2\Biggl(\cfrac{\pi (j-1) p}{n}-\cfrac{\pi (j-1) }{2n}\Biggr)}{\sin^2\left(0.5\pi (j-1)/n \right)\biggl\{\sum_{p=1}^{n}\cos^2\left(\frac{\pi p (j-1)}{n}-\frac{\pi(j-1)}{2n}\right)\biggr\}}+O(\epsilon^3)$ terms.
	
	\noindent Further, Lemma~\ref{lem:perturb:shiftmax:one}, showed that \\
	$\lambda_{min}(\Lmod_{n,\epsilon}, p) = \cfrac{\epsilon}{n}  - \epsilon^2\left[\cfrac{(n^2-1)+12(p-p^*)^2}{12n^2}\right]  +O(\epsilon^3)$ terms.
	
	\noindent On combining the above two expressions of $\lambda_{\min}(\Lmod_n)$, we get the required result: 
	\begin{align*}
		\lambda_{min}(\Lmod_n, p)&=\epsilon \beta_1+\epsilon^2 \beta_2+O(\epsilon^3)=\frac{\epsilon}{n} - \epsilon^2 \left[\frac{(n^2-1)+12(p-p^*)^2}{12n^2}\right]	+O(\epsilon^3)\\
		&=\cfrac{\epsilon}{n}-\cfrac{\epsilon^2}{4n}\sum_{j=2}^{n}\cfrac{\cos^2\Biggl(\cfrac{\pi (j-1) p}{n}-\cfrac{\pi (j-1) }{2n}\Biggr)}{\sin^2\left(0.5\pi (j-1)/n \right)\biggl\{\sum_{p=1}^{n}\cos^2\left(\frac{\pi p (j-1)}{n}-\frac{\pi(j-1)}{2n}\right)\biggr\}}+O(\epsilon^3)		
	\end{align*}
	This completes the part a) of the proof. 
	
	\noindent \textbf{Part~b)} 
	\noindent We have upto $2^{nd}$ order of $\epsilon$:\\ 
	 $\hat{\lambda}_{min}(\Lmod_{n,\epsilon}, p)=\cfrac{\epsilon}{n}-\cfrac{\epsilon^2}{4n}\sum_{j=2}^{n}\cfrac{\cos^2\Biggl(\cfrac{\pi (j-1) p}{n}-\cfrac{\pi (j-1) }{2n}\Biggr)}{\sin^2\left(0.5\pi (j-1)/n \right)\biggl\{\sum_{i=1}^{n}\cos^2\left(\frac{\pi i (j-1)}{n}-\frac{\pi(j-1)}{2n}\right)\biggr\}}$
	
	Derivative of $\hat{\lambda}_{min}(\Lmod_n, p)$ w.r.t. $p$: \vspace{-2mm}
		\begin{equation*}
			\cfrac{d\hat{\lambda}_{min}(\Lmod_n, p)}{dp}=\cfrac{\epsilon^2}{4n}\sum_{j=2}^{n}\cfrac{\sin\Biggl(\cfrac{2\pi (j-1) p}{n}-\cfrac{\pi (j-1) }{n}\Biggr)\Biggl(\cfrac{\pi (j-1)}{n}\Biggr)}{\sin^2\left(0.5\pi (j-1)/n \right)\biggl\{\sum_{i=1}^{n}\cos^2\left(\frac{\pi i (j-1)}{n}-\frac{\pi(j-1)}{2n}\right)\biggr\}}
		\end{equation*}
 Derivative of $\hat{\lambda}_{min}(\Lmod_n, p)$ at $p=\frac{n+1}{2}$:\vspace{-2mm}
		\begin{align*}
			\cfrac{d\hat{\lambda}_{min}(\Lmod_n, p)}{dp}&=\cfrac{\epsilon^2}{4n}\sum_{j=2}^{n}\cfrac{\sin\Biggl(\cfrac{\pi (j-1)(n+1)}{n}-\cfrac{\pi (j-1) }{n}\Biggr)\Biggl(\cfrac{\pi (j-1)}{n}\Biggr)}{\sin^2\left(0.5\pi (j-1)/n \right)\biggl\{\sum_{i=1}^{n}\cos^2\left(\frac{\pi i (j-1)}{n}-\frac{\pi(j-1)}{2n}\right)\biggr\}}\\
			=&\cfrac{\epsilon^2}{4n}\sum_{j=2}^{n}\cfrac{\sin\biggl(\pi (j-1)\biggr)\Biggl(\cfrac{\pi (j-1)}{n}\Biggr)}{\sin^2\left(0.5\pi (j-1)/n \right)\biggl\{\sum_{i=1}^{n}\cos^2\left(\frac{\pi i (j-1)}{n}-\frac{\pi(j-1)}{2n}\right)\biggr\}}=0.
		\end{align*}
		Thus, $p=\frac{n+1}{2}$ is a stationary point. We now check $2^{nd}$ derivative condition. 
		
\noindent Double derivative of $\hat{\lambda}_{min}(\Lmod_n, p)$ w.r.t. $p$: \vspace{-2mm}
		\begin{equation*}
			\cfrac{d^2\hat{\lambda}_{min}(\Lmod_n, p)}{dp^2}=\cfrac{\epsilon^2}{2n}\sum_{j=2}^{n}\cfrac{\cos\Biggl(\cfrac{2\pi (j-1) p}{n}-\cfrac{\pi (j-1) }{n}\Biggr)\Biggl(\cfrac{\pi (j-1)}{n}\Biggr)^2}{\sin^2\left(0.5\pi (j-1)/n \right)\biggl\{\sum_{i=1}^{n}\cos^2\left(\frac{\pi i (j-1)}{n}-\frac{\pi(j-1)}{2n}\right)\biggr\}}.
		\end{equation*}
	Evaluating double derivative of $\hat{\lambda}_{min}(\Lmod_n, p)$ w.r.t. $p$ at $p=\frac{n+1}{2}$: \vspace{-2mm}
		\begin{align*}
			\cfrac{d^2\hat{\lambda}_{min}(\Lmod_n, p)}{dp^2}&=\cfrac{\epsilon^2}{2n}\sum_{j=2}^{n}\cfrac{\cos\Biggl(\cfrac{\pi (j-1) (n+1)}{n}-\cfrac{\pi (j-1) }{n}\Biggr)\Biggl(\cfrac{\pi (j-1)}{n}\Biggr)^2}{\sin^2\left(0.5\pi (j-1)/n \right)\biggl\{\sum_{i=1}^{n}\cos^2\left(\frac{\pi i (j-1)}{n}-\frac{\pi(j-1)}{2n}\right)\biggr\}}\\
			&=\cfrac{\epsilon^2}{2n}\sum_{j=2}^{n}\cfrac{\cos\biggl(\pi (j-1) \Biggr)\Biggl(\cfrac{\pi (j-1)}{n}\Biggr)^2}{\sin^2\left(0.5\pi (j-1)/n \right)\biggl\{\sum_{i=1}^{n}\cos^2\left(\frac{\pi i (j-1)}{n}-\frac{\pi(j-1)}{2n}\right)\biggr\}}
		\end{align*}\vspace{-2mm}
		\begin{equation*}
			\implies	\cfrac{d^2\hat{\lambda}_{min}(\Lmod_n, p)}{dp^2}\Bigg|_{p=\frac{n+1}{2}}=\cfrac{\epsilon^2}{2n}\sum_{j=2}^{n}\cfrac{(-1)^{j-1} \left(\cfrac{\pi (j-1)}{n}\right)^2}{ \sin^2\left(0.5\pi (j-1)/n \right)\biggl\{\sum_{i=1}^{n}\cos^2\left(\frac{\pi i (j-1)}{n}-\frac{\pi(j-1)}{2n}\right)\biggr\}}.
		\end{equation*}
		Let $\mathscr{N}_j=\sum_{i=1}^{n}\cos^2\left(\frac{\pi i (j-1)}{n}-\frac{\pi(j-1)}{2n}\right)$. Since $n$ is odd, the above sum is a sum of  $\frac{n-1}{2}$ pairs of successive differences:  \vspace{-2mm}
		\begin{equation}\label{eq:dd_p}
			\cfrac{d^2\hat{\lambda}_{min}(\Lmod_n, p)}{dp^2}=\cfrac{2\epsilon^2}{n}\sum_{q=1}^{\frac{n-1}{2}}\left\{\cfrac{\left(\cfrac{\pi 2q}{2n}\right)^2 }{\mathscr{N}^2_{2q}\sin^2\left(\cfrac{\pi 2q}{2n}\right)} \right \}-\left \{\cfrac{\left(\cfrac{\pi (2q-1)}{2n}\right)^2 }{\mathscr{N}^2_{2q-1}\sin^2\left(\cfrac{\pi (2q-1)}{2n}\right)} \right \}.\\
		\end{equation}
		Each pair is positive since for $q \in \{1, 2, \dots \frac{n-1}{2}\}$, $\mathscr{N}$ is constant and since $\frac{x}{\sin(x)}$ is monotonically increasing for $x \in (0, \frac{\pi}{2})$ \footnote{We use the fact that the $\cfrac{x}{\sin(x)}$ is a monotonically increasing function in the interval $[0, \frac{\pi}{2})$ and that this expression is of the form $\cfrac{x_2^2}{\sin^2(x_2)}<\cfrac{x_1^2}{\sin^2(x_1)}$ for $0<x_1<x_2<\frac{\pi}{2}$.}.
%
		\noindent Thus, we conclude that $\cfrac{d^2\hat{\lambda}_0}{dp^2}>0$, for $p=p^*$.
		 
%
%

	\noindent Therefore, $p^*$ is a local minima (considering $p\in\mathbb{R}$). This completes the part b) of the proof and thus this completes the proof of Lemma~\ref{lem:1port_path_EigenShift}.	
\end{proof} 

\begin{proof}[Proof of Lemma~\ref{lem:2ports_path_EigenShift}:]
%
	\noindent From Lemma~\ref{lem:kports_path_EigenShift}, we have: \vspace{-4mm} \begin{small}\[\lambda_{min}(^S\!\Lmod_n)=\cfrac{k\epsilon}{n}-\cfrac{\epsilon^2}{4n}\sum_{j=2}^{n}\cfrac{\Biggl\{\sum_{i=1}^{k}\cos\biggl(\cfrac{\pi (j-1) p_i}{n}-\cfrac{\pi (j-1) }{2n}\biggr)\Biggr\}^2}{\sin^2\left(0.5\pi (j-1)/n \right)\left\{\sum_{i=1}^{n}\cos^2\left(\pi  (j-1)(i-0.5)/n\right)\right\}}+O(\epsilon^3)\] 
	\end{small}

\noindent So, for $2$-ports selection, using $k=2$, in above equation and we get: \vspace{-2mm}\\ 
\scalebox{0.9}{$\lambda_{min}(\Lmodd_{n,\epsilon}, p_1, p_2)=\cfrac{2\epsilon}{n}-\cfrac{\epsilon^2}{4n}\sum_{j=2}^{n}\cfrac{\Biggl\{\cos\biggl(\cfrac{\pi (j-1) p_1}{n}-\cfrac{\pi (j-1) }{2n}\biggr)+\cos\biggl(\cfrac{\pi (j-1) p_2}{n}-\cfrac{\pi (j-1) }{2n}\biggr)\Biggr\}^2}{\sin^2\left(0.5\pi (j-1)/n \right)\left\{\sum_{i=1}^{n}\cos^2\left(\pi  (j-1)(i-0.5)/n\right)\right\}}+O(\epsilon^3)$}

	\noindent Further, from Lemma~\ref{lem:perturb:shiftmax:two}, it follows that\\
	 $\lambda_{min}(\Lmodd_{n,\epsilon},p_1, p_2) \approx \cfrac{2\epsilon}{n}  -\epsilon^2 \left[\cfrac{n^2-4}{12n^2}+2\left\{\cfrac{(p_1-p_1^*)^2+(p_2-p_2^*)^2}{n^2}\right\} \right]$
	 
	\noindent $\implies 
	 \lambda_{min}(\Lmodd_{n,\epsilon}, p_1, p_2) = \cfrac{2\epsilon}{n}  -\epsilon^2 \left[\cfrac{n^2-4}{12n^2}+2\left\{\cfrac{(p_1-p_1^*)^2+(p_2-p_2^*)^2}{n^2}\right\} \right]  +O(\epsilon^3)$ 
	 
	\noindent This completes the proof. 	
\end{proof} 

\begin{proof}[Proof of Lemma~\ref{lem:kports_path_EigenShift}:]
	\noindent We know for path graph from Proposition~\ref{prop:pathEigvalEigvec}: $\mbox{for}~j=\{1,2, \dots,n\}$\\	
	\noindent The ordered eigenvalue, $\lambda_j=2\Biggl\{1-\cos\biggl(\cfrac{\pi (j-1)}{n}\biggr)\Biggr\}$\\
	\noindent The corresponding eigenvector, $v_j(p)=\cos\biggl(\cfrac{\pi p (j-1)}{n}-\cfrac{\pi (j-1) }{2n}\biggr)$\\
	The normalized eigenvector, \\$\hat{v}_j(p)=\cos\left(\cfrac{\pi p (j-1)}{n}-\cfrac{\pi (j-1) }{2n}\right)\bigg/\sqrt{\sum_{p=1}^{n}\cos^2\left(\cfrac{\pi p (j-1)}{n}-\cfrac{\pi(j-1)}{2n}\right)}$. \\
	Let $\mathscr{N}_j=\sqrt{\sum_{p=1}^{n}\cos^2\left(\cfrac{\pi p (j-1)}{n}-\cfrac{\pi(j-1)}{2n}\right)}$
	$\implies \hat{v}_j(p)=\cos\left(\cfrac{\pi p (j-1)}{n}-\cfrac{\pi (j-1) }{2n}\right)/\mathscr{N}_j$ 
	
	\noindent Using the expressions of $v_j$ and $\lambda_j$ in equation \eqref{eq:perturbed:eigenvector} of Proposition~\ref{prop:eigenvecExpan:stackexchange} (specialized to $L_n$ of path graph), we get the perturbed eigenvector up-to $2^{nd}$ order in $\epsilon$: \\
	$\hat{\tilde{v}}_1(\epsilon)=\hat{v}_1+\epsilon \sum_{j=2}^{n}\cfrac{\cfrac{1}{\sqrt{n}}\sum_{i=1}^{k}\left\{\cos\left(\cfrac{\pi (j-1) p_i}{n}-\cfrac{\pi (j-1) }{2n}\right)\bigg/\mathscr{N}_{j_i}\right\}}{-4\sin^2\left(0.5\pi (j-1)/n \right)}\hat{v}_j+\epsilon^2\psi=\hat{v}_1+\epsilon\Cchi+\epsilon^2\psi\quad$ \\	where $\Cchi=\sum_{j=2}^{n}\cfrac{\cfrac{1}{\sqrt{n}}\sum_{i=1}^{k}\left\{\cos\left(\cfrac{\pi (j-1) p_i}{n}-\cfrac{\pi (j-1) }{2n}\right)\bigg/\mathscr{N}_{j_i}\right\}}{-4\sin^2\left(0.5\pi (j-1)/n \right)}\hat{v}_j$ with $\Cchi, \psi \in \Rn$ \\
	Similarly, consider the perturbed eigenvalue up-to $2^{nd}$ order in $\epsilon$:  $\tilde{\lambda}_1=\lambda_{1}+\epsilon \beta_{1}+\epsilon^2 \beta_{2}$.\\
	Perturbed Laplacian matrix, $\Lmod_n=L_n+\epsilon V=L_n+\epsilon \sum_{i=1}^{k}e_{p_i} e_{p_i}^T$\\
	\noindent Eigenvalue, eigenvector equation for perturbed Laplacian matrix we have: \vspace{-2mm}
	\begin{center}
		$\Lmod_n \hat{\tilde{v}}_1(\epsilon)=\tilde{\lambda}_{1}\hat{\tilde{v}}_1(\epsilon)$\\
	\end{center}\vspace{-2mm}
	By substituting the values of $\Lmod_n$, $\hat{\tilde{v}}_1(\epsilon)$, and $\tilde{\lambda}_{1}$, we obtain: \vspace{-2mm}	
	\begin{equation}\label{eq:qPerturbed_EvecEval}
		(L_n+\epsilon V)(\hat{v}_1+\epsilon\Cchi+\epsilon^2\psi)=(\lambda_{1}+\epsilon \beta_{1}+\epsilon^2 \beta_{2})(\hat{v}_1+\epsilon\Cchi+\epsilon^2\psi).
	\end{equation}
	Upon simplifying the above equation:\vspace{-2mm}
	\begin{center}
		$L_n\hat{v}_1+\epsilon L_n \Cchi+\epsilon^2 L_n \psi+\epsilon V \hat{v}_1+\epsilon^2 V \Cchi +\epsilon^3 V\psi=\lambda_{1}\hat{v}_1+\epsilon\lambda_{1}\Cchi+\epsilon^2\lambda_{1}\psi+\epsilon\beta_{1}\hat{v}_1+\epsilon^2\beta_{1}\Cchi+\epsilon^3\beta_{1}\psi+\epsilon^2\beta_{2}\hat{v}_1+\epsilon^3\beta_{2}\Cchi+\epsilon^4\beta_{2}\psi$
	\end{center}
	On comparing `$\epsilon$' terms: 
		$L_n\Cchi+V\hat{v}_1=\lambda_1 \Cchi+\beta_{1}\hat{v}_1$. Since, $\lambda_{1}(L_n)=0$ and $\hat{v}^T_1 L_n=0$.\\
		On pre-multiplying both sides by $\hat{v}^T_1$, we get \vspace{-2mm}
	\begin{equation}\label{eq:k-ports:epsilonTerm}
		\beta_{1}=\hat{v}^T_1 V \hat{v}_1 =\hat{v}^T_1 \left(\sum_{i=1}^{k}e_{p_i} e_{p_i}^T\right) \hat{v}_1~~ \implies~~ \mathbf{\beta_{1}=\cfrac{k}{n}}
	\end{equation}	
	Similarly, on comparing `$\epsilon^2$' terms: 
	$L_n\psi+(\sum_{i=1}^{k}e_{p_i} e_{p_i}^T)\Cchi=\lambda_{1}\psi+\beta_{1} \Cchi+\beta_{2} \hat{v}_1$. \\
	Since, for symmetric matrix eigenvector corresponding to distinct eigenvalues are orthogonal,  $\hat{v}^T_1 v_{j}=0$ for $j=\{2,3,\dots, n\}$ $\implies\hat{v}^T_1 \Cchi=0$. Also, $\lambda_{1}(L_n)=0$ and $\hat{v}^T_1 L_n=0$. \\
	On pre-multiplying both sides by $\hat{v}^T_1$, we get \vspace{-2mm}
	\begin{equation*}
	\hat{v}^T_1(\sum_{i=1}^{k}e_{p_i} e_{p_i}^T)\Cchi=\hat{v}^T_1 \beta_{2} \hat{v}_1~~ \implies \beta_{2}=\hat{v}^T_1(\sum_{i=1}^{k}e_{p_i} e_{p_i}^T)\Cchi
	\end{equation*}
	\noindent Using the value of $\Cchi$ in above equation, we get: \\
	$\beta_{2}=\hat{v}^T_1(\sum_{i=1}^{k}e_{p_i} e_{p_i}^T)\sum_{j=2}^{n}\cfrac{\cfrac{1}{\sqrt{n}}\left\{\sum_{i=1}^{k}\cos\left(\cfrac{\pi (j-1) p_i}{n}-\cfrac{\pi (j-1) }{2n}\right)\bigg/ \mathscr{N}_{j_i}\right\}}{-4\sin^2\left(0.5\pi (j-1)/n \right)}\hat{v}_j$\\
	
	\noindent Using value of $\hat{v}_1$, we get: \\ 
	$ \beta_{2}=\cfrac{1}{\sqrt{n}}(\sum_{i=1}^{k} e_{p_i}^T)\sum_{j=2}^{n}\cfrac{\cfrac{1}{\sqrt{n}}\left\{\sum_{i=1}^{k}\cos\left(\cfrac{\pi (j-1) p_i}{n}-\cfrac{\pi (j-1) }{2n}\right)\bigg/ \mathscr{N}_{j_i}\right\}}{-4\sin^2\left(0.5\pi (j-1)/n \right)}\hat{v}_j$\\
	On simplifying and using value of $\hat{v}_j$, we get: 	\vspace{-4mm}
	\begin{equation}\label{eq:k-ports:epsilon2Term}		 
	\beta_{2}=\cfrac{1}{n}\sum_{j=2}^{n}\cfrac{\left\{\sum_{i=1}^{k}\cos\left(\cfrac{\pi (j-1) p_i}{n}-\cfrac{\pi (j-1) }{2n}\right)\bigg/ \mathscr{N}_{j_i}\right\}^2}{-4\sin^2\left(0.5\pi (j-1)/n \right)}
	\end{equation}
	Hence, $\lambda_{min}(\Lmod_{n,\epsilon}, p)=\tilde{\lambda}_1=\lambda_{1}+\epsilon \beta_{1}+\epsilon^2 \beta_{2}$\\
	Using values of $\beta_{1}$ and $\beta_{2}$ from equation~\eqref{eq:k-ports:epsilonTerm},~\eqref{eq:k-ports:epsilon2Term} and $\lambda_{1}(L_n)=0$, we get:\vspace{-3mm} \\
	\begin{equation*}
		\lambda_{min}(\Lmod_{n,\epsilon}, p)=\cfrac{k\epsilon}{n}-\cfrac{\epsilon^2}{4n}\sum_{j=2}^{n}\cfrac{\left\{\sum_{i=1}^{k}\cos\left(\cfrac{\pi (j-1) p_i}{n}-\cfrac{\pi (j-1) }{2n}\right)\bigg/ \mathscr{N}_{j_i}\right\}^2}{\sin^2\left(0.5\pi (j-1)/n \right)}
	\end{equation*} \\
	This completes the proof of Lemma~\ref{lem:kports_path_EigenShift}.	
\end{proof} 
\begin{proof}[Proof of Lemma~\ref{lem:perturbedLvsN}:]
	\noindent Consider $\Lmod_n:= L_n + e_1e_1^T$ which is a tridiagonal pseudo-Toeplitz\footnote{Pseudo-Toeplitz matrix: Matrix $M$ in Proposition~\ref{prop:characteristic:Tri}: $a_1=a_2=\cdots=a_{m-1}=2$, $a_m=1$ and $b_i=c_i=-1$.} matrix. 
	
	\noindent Using the result from \cite{KulkarniEtal}, we get
	$\lambda(\Lmod_n)=2-2\cos\Bigl(\cfrac{\pi k}{2n+1}\Bigr)$,  where $k=1,2,\dots, n$. \\
	So, for $k=1$, $\lambda_{1}(\Lmod_n)=\lambda_{\min}(\Lmod_n)=2-2\cos\Bigl(\cfrac{\pi }{2n+1}\Bigr)$.\\
	Consider,  $\Lmodd_{2n}=L_{2n}+e_1e_1^T+e_{2n}e_{2n}^T$ which is a tridiagonal Toeplitz\footnote{Pseudo-Toeplitz matrix: Matrix $M$ in Proposition~\ref{prop:characteristic:Tri}: $a_1=a_2=\cdots=a_{m-1}=a_m=2$, and $b_i=c_i=-1$.} matrix.
	
	\noindent Hence, again from \cite{KulkarniEtal}, we get	 
	$\lambda(T_n)=2-2\cos\Bigl(\cfrac{\pi k}{n+1}\Bigr)~~\mbox{where}~~k=1,2,\dots, n $
	
	\noindent Therefore for $2n\times 2n$: 
	$\lambda(\Lmodd_{2n})=2-2\cos\Bigl(\cfrac{\pi k}{2n+1}\Bigr)~~ \mbox{where}~~k=1,2,\dots, 2n$
	
	\noindent So, for $k=1$,  $\lambda_{1}(\Lmodd_{2n})=\lambda_{min}(\Lmodd_{2n})=2-2\cos\Bigl(\cfrac{\pi}{2n+1}\Bigr)=\lambda_{min}(\Lmod_{n})$. \\
	Hence, $\lambda_{\min}(\Lmod_n)$ and $\lambda_{min}(\Lmodd_{2n})$ are equal. 
	
	\noindent \textbf{To show monotonically decreasing}: \\
	On differentiating $\lambda_{\min}(\Lmod_n)=2-2\cos\Bigl(\cfrac{\pi }{2n+1}\Bigr)$ w.r.t. $n$, we get\\
	$\cfrac{\partial\lambda_{\min}(\Lmod_{n})}{\partial n}=-\cfrac{4\pi}{(2n+1)^2}\sin\Bigl(\cfrac{\pi}{2n+1}\Bigr)$ $\implies \cfrac{\partial\lambda_{\min}(\Lmod_{n})}{\partial n}<0$. 
	
	\noindent Hence, $\lambda_{\min}(\Lmod_n)$ is monotonically decreasing w.r.t. $n$.\\
	This completes the proof of Lemma~\ref{lem:perturbedLvsN}.	
\end{proof}

\begin{proof}[Proof of Lemma~\ref{lem:perturb:epsi:num:den:series}:]
	Consider the polynomials 
	$a(s) = a_0 + a_1s + a_2 s^2+ \cdots +s^n$ and $b(s) = b_0 + b_1s + b_2 s^2  + \cdots+ b_{n-1}s^{n-1}$ be two finite-degree polynomials with $a_0 = 0$,  $a_1 \ne 0$, and $b_0 \ne 0$.\\
	Define the polynomial $p_\epsilon(s): = a(s) + \epsilon b(s)$. 
	Clearly $s=0$ is a root of $p_\epsilon(s)$ for $\epsilon = 0$.\\ 
	In addition consider the dependence $\lambda_{min}$ on $\epsilon$; 
	expand $\lambda_{min}(\epsilon)$ about the origin in a series expansion:\\
	$\lambda_{min}(\epsilon) = \beta_1 \epsilon + \beta_2 \epsilon^2 + \beta_3 \epsilon^3 + \cdots$.
	
	\noindent Rearrange $a(s) + \epsilon b(s)=0$, to get $\epsilon(s)  = \cfrac{a(s)}{b(s)}$, and to make the dependence of $\epsilon$ on $s$ more explicit, express\\
	$\epsilon (s)  = \cfrac{-a(s)}{b(s)}=\alpha_1 s+\alpha_2 s^2$.\\
	\noindent Evaluating $\displaystyle \frac{d\epsilon}{ds}$ at $s=0$, we
	obtain:
	$\displaystyle \frac{d\epsilon}{ds}\bigg|_{s=0} =\alpha_1= \cfrac{-a_1}{b_0}$
	
	\noindent Evaluating $\displaystyle \frac{d^2\epsilon}{ds^2}$ at $s=0$, we obtain: 
	$\displaystyle \frac{d^2\epsilon}{ds^2}\bigg|_{s=0} =2\alpha_2= \cfrac{2a_1 b_1 - 2a_2 b_0 }{b_0^2 }$ $\implies ~\alpha_2= \cfrac{a_1 b_1 - a_2 b_0 }{b_0^2 }$
	
	\noindent Further, consider, $s=\lambda_{min}(\epsilon) \approx \beta_1 \epsilon+\beta_2 \epsilon^2$.
	
	\noindent Using standard formulae (see~\cite{Weisstein}) that relate the coefficients in 
	the Taylor series expansions of $f(\cdot)$ and $f^{-1}(\cdot)$, we get the desired relations in equation~\eqref{eq:taylor:inverse}.
	$\beta_1=\displaystyle\frac{1}{\alpha_1}$ and $\beta_2=\displaystyle-\frac{\alpha_2}{\alpha_1^3}$\\
	Therefore,
	$\beta_1=\cfrac{-b_0}{a_1}$ and $\beta_2=\cfrac{a_1 b_1 - a_2 b_0 }{b_0^2 }\times \cfrac{b_0^3}{a_1^3}=\cfrac{a_1 b_1b_0 - a_2 b^2_0 }{a_1^3 }$ \\
	This completes the proof.
\end{proof}

\begin{proof}[Proof of Lemma~\ref{lem:perturb:doubleepsi:num:den:series}:]
	Consider polynomials $a(s) = a_0 + a_1s + a_2 s^2+ \cdots +s^n$, $b(s) = b_0 + b_1s + b_2 s^2  + \cdots+ b_{n-1}s^{n-1}$
	and $c(s) = c_0 + c_1s + c_2 s^2  + \cdots+ c_{n-2}s^{n-2}$ such that\\
	$a_0 = 0$,  $a_1 \ne 0$, $b_0 \ne 0$ and $c_0 \ne 0$.
	Define the polynomial $p_\epsilon(s): = a(s) + \epsilon b(s)+ \epsilon^2 c(s)$. 
	
	\noindent In addition consider the root $\lambda_{min}$ and
	the dependence of this root on $\epsilon$. Further, note $\lambda_{min}(0)=0$.
	
	\noindent Expand $\lambda_{min}(\epsilon)$ about the origin in a series expansion:\\
	$\lambda_{min}(\epsilon) = \beta_1 \epsilon + \beta_2 \epsilon^2 + \beta_3 \epsilon^3 + \cdots$.
	
	\begin{equation}\label{eq:epsilon:double}
	\mbox{Consider a series expansion of }\epsilon(s) \mbox{ about } \epsilon=0 \hspace*{2cm}\epsilon(s)  = \alpha_1 s+\alpha_2 s^2+\alpha_3 s^3+ \cdots. \hspace*{1.5cm}
	\end{equation}
	On differentiating $p_{\epsilon}(s)$ w.r.t. `$s$', we get: \\
	$p'_{\epsilon}(s)=0$ 	$\implies a'+\epsilon b'+\epsilon^2 c' +\epsilon'b+2\epsilon \epsilon'c=0$ 
	$\implies \epsilon'= \cfrac{-(a'+ \epsilon b' + \epsilon^2 c')}{b +2 \epsilon c}$
	
	\noindent We know from above expression and \eqref{eq:epsilon:double}, $\alpha_1=\epsilon'(0)=\cfrac{-a_1}{b_0}$
	
	\noindent Further, on differentiating $p_{\epsilon^2}(s)$ twice w.r.t. `$s$', we get: 
	$a''+\epsilon b''+\epsilon'' b+2\epsilon' b'+4 \epsilon \epsilon' c'+2(\epsilon')^{2}c+\epsilon^2 c''+2\epsilon \epsilon'' c=0$\\
	$\implies \epsilon''(0)=2\cfrac{a_1b_1b_0-a_2b_0^2-a_1^2c_0}{b_0^3}$\\
	From \eqref{eq:epsilon:double}, $\alpha_2=\cfrac{\epsilon''(0)}{2}=\cfrac{a_1b_1b_0-a_2b_0^2-a_1^2c_0}{b_0^3}$
	
	\noindent Now consider, $s=\lambda_{min}(\epsilon) \approx \beta_1 \epsilon+\beta_2 \epsilon^2$.
	
	\noindent Using standard formulae (see~\cite{Weisstein})
	that relate the coefficients in 
	the Taylor series expansions of $f(\cdot)$ and $f^{-1}(\cdot)$, we get the desired relations in equation~\eqref{eq:taylor:doubleinverse}. 
	$\beta_1=\cfrac{1}{\alpha_1}$ and $\beta_2=-\cfrac{\alpha_2}{\alpha_1^3}$\\
	Therefore, using values of $\alpha_i$, we get,\\ $\beta_1=\cfrac{-b_0}{a_1}$ and  $\beta_2=\cfrac{(a_1b_1b_0-a_2b_0^2-a_1^2c_0)}{b_0^3}\times \cfrac{b_0^3}{a_1^3}=\cfrac{(a_1b_1b_0-a_2b_0^2-a_1^2c_0)}{a_1^3}$\\
	This completes the proof of Lemma~\ref{lem:perturb:doubleepsi:num:den:series}.
\end{proof}

\begin{proof}[Proof of Lemma~\ref{lem:perturb:shiftmax:one}:]
%
	\noindent \textbf{Part~a)} 
%
	\noindent Using the results from Lemma~\ref{lem:perturb:epsi:num:den} and Proposition~\ref{prp:characteristic:path:explicit} for $n$ odd, we have:\\
	 $a_1=n,~~a_2=-\cfrac{n(n^2-1)}{6}$, 
	$b_0=-1,	~~b_1=\{\cfrac{(n-1)^2 + 2(n-1)}{4} + (j-p^*)^2\}$.
	 
	\noindent Further from equation \eqref{eq:taylor:inverse} of Lemma~\ref{lem:perturb:epsi:num:den:series}, we have\\ $\beta_1=\cfrac{-b_0}{a_1}=\cfrac{1}{n}$ and  $\beta_2=\cfrac{a_1b_1b_0-a_2b_0^2}{a_1^3}=\cfrac{-n\{\frac{(n-1)^2+2(n-1)}{4}+(j-p^*)^2\}+\frac{n(n^2-1)}{6}}{n^3}$
	
	\noindent $\implies \beta_2=\cfrac{n^2-1}{6n^2}+\cfrac{\frac{(n-1)^2 + 2(n-1)}{4} + (j-p^*)^2}{n^2}=-\{\cfrac{(n^2-1)+12(j-p^*)^2}{12n^2}\}$
	
	\noindent Thus, $\lambda_{min}$ upto $2^{nd}$ in $\epsilon$, is\vspace{-2mm}
	\begin{equation*}
		\lambda_{min}(\epsilon,j) = \beta_1 \epsilon + \beta_2 \epsilon^2 
	\end{equation*}
	Using the values of $\beta_1,~\beta_2$ in above equation, we get:
	\begin{equation}\label{eq:lambda_0:approx:two}
		\lambda_{min}(\epsilon,j) = \frac{\epsilon}{n}  - \epsilon^2\left[\frac{(n^2-1)+12(j-p^*)^2}{12n^2}\right]  
	\end{equation}
	This completes the proof of part (a).\\
	
	\noindent \textbf{Part b)}
\noindent $\lambda_{min}(\epsilon,j) = \cfrac{\epsilon}{n}  - \epsilon^2\left[\cfrac{(n^2-1)+12(j-p^*)^2}{12n^2}\right] +O(\epsilon^3)$ is quadratic in $j$ (after ignoring terms of $\epsilon^3$ and higher) and this curve is a parabola that clearly peaks at $j=p^*$. 
%
	\noindent Thus $j=p^*$ globally maximizes the $\lambda_{min}(\epsilon,j)$ and the corresponding maximum value (upto $2^{nd}$ order in $\epsilon$) is\\ $\lambda_{min}(\epsilon,p^*) = \cfrac{\epsilon}{n}  - \epsilon^2\left[\cfrac{n^2-1}{12n^2}\right] $ 
	
%

	This completes the proof of part b) and proof of Lemma~\ref{lem:perturb:shiftmax:one}.
\end{proof}

\begin{proof}[Proof of Lemma~\ref{lem:perturb:shiftmax:two}:]
%
 	\noindent \textbf{Part~a) }
	\noindent Using results from Lemmas \ref{lem:perturb:epsi:num:den:two} and Proposition~\ref{prp:characteristic:path:explicit} for $n$ even, we have:\\ $a_1=-n,~~a_2=\cfrac{n(n^2-1)}{6},~~c_0=(j_2-j_1)$\\
	$b_0=2,	~~b_1=-\biggl(\cfrac{3n^2-4}{8} -\cfrac{n(j_1-j_2)}{2}+(j_1-p^*_1)^2+(j_2-p^*_2)^2\biggr)$.\\ 
	Further from equation \eqref{eq:taylor:doubleinverse} of Lemma~\ref{lem:perturb:doubleepsi:num:den:series} we have $\beta_1=\cfrac{-b_0}{a_1}=\cfrac{-2}{-n}=\cfrac{2}{n}$ and \\
		 $\beta_2=\cfrac{b_1b_0}{a_1^2}-\cfrac{a_2b^2_0}{a_1^3}-\cfrac{c_0}{a_1}		 =\cfrac{n(j_1-j_2)-2(j_1-p_1^*)^2-2(j_2-p_2^*)^2-\frac{3n^2-4}{4}}{n^2}+\cfrac{2(n^2-1)}{3n^2}+\cfrac{j_2-j_1}{n}$\\
		 $\implies \beta_2=\left[\cfrac{n^2-4}{-12n^2}-2\left\{\cfrac{(j_1-p_1^*)^2+(j_2-p_2^*)^2}{n^2}\right\} \right]$
		 
	\noindent Thus, $\lambda_{min}$ upto $2^{nd}$ order in $\epsilon$, is\vspace{-2mm}
	\begin{equation*}
		\lambda_{min}(\epsilon,j_1, j_2)=\beta_1 \epsilon + \beta_2 \epsilon^2 
	\end{equation*}
	Using the values of $\beta_1,~\beta_2$ in above equation, we get:\vspace{-1mm}
	\begin{equation}\label{eq:lambda_0:approx:two}
		\lambda_{min}(\epsilon, j_1, j_2) = \cfrac{2\epsilon}{n}  -\epsilon^2 \left[\cfrac{(n^2-4)+24\left\{(p_1-p_1^*)^2+(p_2-p_2^*)^2\right\}}{12n^2} \right] 
	\end{equation}
	
	\noindent \textbf{Part~b)} 
%
	\noindent $\lambda_{min}(j_1, j_2) = \cfrac{2\epsilon}{n}  -\epsilon^2 \left[\cfrac{(n^2-4)+24\left\{(p_1-p_1^*)^2+(p_2-p_2^*)^2\right\}}{12n^2} \right] +O(\epsilon^3)$ is quadratic in $j_1$ and $j_2$ (after ignoring terms of $\epsilon^3$ and higher) and this curve clearly peaks at $j_1=p_1^*$ and $j_2=p_2^*$. \\
	\noindent Thus $j_1=p_1^*$, $j_2=p_2^*$,  maximizes the $\lambda_{min}(\epsilon, j_1, j_2)$ and 
%
%
%
the maximum value is (upto $2^{nd}$ order in $\epsilon$): \vspace{-2mm}
	 \begin{equation*}
	 	\lambda_{min}(\epsilon, p_1^*, p_2^*) = \frac{2\epsilon}{n} -\epsilon^2 \left[\frac{n^2-4}{12n^2}\right]
	 \end{equation*}
	 This complete the proof of part b) and proof of Lemma~\ref{lem:perturb:shiftmax:two}.
\end{proof}

\end{document}